\documentclass[12pt]{amsart}
\usepackage{amssymb}
\usepackage{graphicx}
\usepackage[hypertex]{hyperref}
\allowdisplaybreaks[1]

\newcommand{\C}{\mathbb C}

\newcommand{\N}{\mathbb N}

\renewcommand{\H}{\mathbb H}
\renewcommand{\P}{\mathbb P}

\newcommand{\Z}{\mathbb Z}

\newcommand{\cE}{\mathcal E}
\newcommand{\CC}{\mathcal C}
\newcommand{\cF}{\mathcal F}

\newcommand{\cH}{\mathcal H}

\newcommand{\cM}{\mathcal M}
\newcommand{\cN}{\mathcal N}

\newcommand{\cx}{\tilde x}
\newcommand{\fX}{\frak{X}}
\newcommand{\re}{\mathrm{Re}}
\newcommand{\im}{\mathrm{Im}}
\newcommand{\Sol}{\mathrm{Sol}}

\newcommand{\na}{\nabla}
\newcommand{\pa}{\partial}
\newcommand{\ex}{\mathbf{e}}
\newcommand{\cex}{\tilde{\mathbf{e}}}
\renewcommand{\a}{\alpha}
\renewcommand{\b}{\beta}
\newcommand{\g}{\gamma}

\renewcommand{\l}{\lambda}

\newcommand{\cL}{\mathcal{L}}

\newcommand{\e}{\varepsilon}
\newcommand{\f}{\varphi}

\newcommand{\G}{\mathit{\Gamma}}

\newcommand{\w}{\omega}

\newcommand{\reg}{\mathrm{reg}}

\newcommand{\la}{\langle}
\newcommand{\ra}{\rangle}
\newcommand{\tr}{\;^t}

\newcommand{\bv}{\mathbf{v}}
\newcommand{\bw}{\mathbf{w}}

\newcommand{\diag}{\mathrm{diag}}

\newcommand{\pr}{\mathrm{pr}}
\newcommand{\comment}[1]{}

\newcommand{\tH}[1]{H_1(T_x,\cL_x^{#1})}

\newcommand{\lftH}[1]{H_1^{lf}(T_x,\cL_x^{#1})}
\newcommand{\tC}[1]{H^1(T_x,\cL_x^{#1})}
\newcommand{\ctC}[1]{H^1_{c}(T_x,\cL_x^{#1})}
\newcommand{\LS}[1]{\cH_1(\cL^{#1})}
\newcommand{\lfLS}[1]{\cH_1^{lf}(\cL^{#1})}
\newcommand{\SV}[1]{V(\cL^{#1})}
\newcommand{\lfSV}[1]{V^{lf}(\cL^{#1})}
\newcommand{\dfc}{(\f_c)^*}

\title[Monodromy of local systems associated with $F_D$]
{The monodromy representations of local systems
associated with Lauricella's $F_D$ 
}
\author{Keiji Matsumoto}
\address[Matsumoto]{
   Department of Mathematics,
   Hokkaido University,
   Sapporo 060-0810, Japan
}
\email{matsu@math.sci.hokudai.ac.jp}

\date{\today}

\theoremstyle{plain} %text of this environment is typesetted in italics
\newtheorem{theorem}{\indent\sc Theorem}[section]
\newtheorem{lemma}[theorem]{\indent\sc Lemma}
\newtheorem{cor}[theorem]{\indent\sc Corollary}
\newtheorem{proposition}[theorem]{\indent\sc Proposition}

\theoremstyle{definition} %text of this environment is typesetted in roman letters
\newtheorem{definition}[theorem]{\indent\sc Definition}
\newtheorem{fact}[theorem]{\indent\sc Fact}
\newtheorem{remark}[theorem]{\indent\sc Remark}

\numberwithin{equation}{section}
\begin{document}
\maketitle
%\dedicatory
%{\textsl{Dedicated to Professor XXX YYY on his sixtieth birthday}}
\begin{abstract}
We give the monodromy representations of local systems of twisted homology 
groups associated with Lauricella's  system $\cF_D(a,b,c)$ of 
hypergeometric differential equations under mild conditions on parameters. 
Our representation is effective even in some cases when 
the system $\cF_D(a,b,c)$ is reducible. 
We characterize invariant subspaces under our monodromy representations  
by the kernel or image of a natural map from a finite twisted homology 
group to locally finite one.
\end{abstract}

\section{Introduction}
\label{sec:Intro}

There are several generalizations of 
%the hypergeometric series and 
the hypergeometric equation. 
Lauricella's system $\cF_D(a,b,c)$ of 
hypergeometric differential equations is regarded as the simplest system 
with multi-variables. 
It is regular singular and its rank is one more than the number of variables. 
Its singular locus $S$  
%for $m$-variables
 is given in (\ref{eq:sing-loc}) and 
the fundamental group of its complement can be interpreted by 
the pure braid group.  
The monodromy representation of $\cF_D(a,b,c)$ is studied by 
several authors under the non-integral condition on parameters:
\begin{equation}
\label{eq:old-condition}
\a=(\a_0,\a_1,\dots,\a_m,\a_{m+1},\a_{m+2})=
(-c+\sum_{i=1}^m b_i,-b_1,\dots,-b_m,c-a,a)
\in (\C-\Z)^{m+3},
\end{equation}
where $m$ is the number of variables and the entries  satisfy 
\begin{equation}
\label{eq:cond-wa}
\sum_{i=0}^{m+2} \a_i=0;
\end{equation}
cf. \cite{DM}, \cite{IKSY}, \cite{OT}, \cite{M2}, \cite{T} 
and the references therein. 
As one of them, it is shown in \cite[Theorem 5.1]{M2} that   
circuit transformations are represented by the 
intersection form between twisted homology groups associated with  
the integral representation (\ref{eq:intrep}) of Euler type.  
The results in \cite{M2} are based on a fact that the trivial vector bundle  
$\bigcup_{x\in U_x} \tH{\a}$ is isomorphic to 
the local solution space to $\cF_D(a,b,c)$ on $U_x$, where 
$x\in X=(\P^1)^m-S$, $U_x(\subset X)$ is a simply connected small 
neighborhood of $x$, and $\tH{\a}$ is the twisted homology group 
%with respect to $u(t,x)$   
(refer to (\ref{eq:t-H}) for its definition). 
% The monodromy representation of $\cF_D(a,b,c)$ coincides with 
% that of the local system $\LS{\a}=\bigcup_{x\in X} \tH{\a}$ over $X$. 

In this paper, we generalize \cite[Theorem 5.1]{M2} by relaxing 
the condition (\ref{eq:old-condition}) to 
\begin{equation}
\label{eq:non-int}
(\a_0,\a_1,\dots,\a_m,\a_{m+1},\a_{m+2})\notin \Z^{m+3}.
\end{equation}
Note that there are at least two entries 
%$\a_{i_{m+1}}$ and$\a_{i_{m+2}}$ of $\a$ such that 
$\a_{i_{m+1}},\a_{i_{m+2}}\notin \Z$ by (\ref{eq:cond-wa}). 
Under this condition,   
we study the monodromy representations $\cM^\a$ and $\cN^{-\a}$ 
of local systems $\lfLS{\a}$ and $\LS{-\a}$ with fibers $\lftH{\a}$ 
and $\tH{-\a}$, respectively, where 
$\lftH{\a}$ is the locally finite twisted homology group and 
$\tH{-\a}$ is given by the sign change $\a\mapsto -\a$ for $\tH{\a}$.
There is a natural linear map $\jmath_h^\a$ from $\tH{\a}$ to $\lftH{\a}$. 
It is known that this map is isomorphic  under 
the condition (\ref{eq:old-condition}). However, if there is an entry 
$\a_i$ of $\a$ such that $\a_i\in \Z$, then both of 
the kernel and the image of $\jmath_h^\a$ are proper subspaces.
Thus it turns out that the monodromy representations $\cM^\a$ and $\cN^{-\a}$ 
of $\lfLS{\a}$ and 
$\LS{-\a}$ are reducible in this case. 
In spite of this situation, the intersection form 
$\la\;,\;\ra$ between $\lftH{\a}$ and $\tH{-\a}$ is well-defined and perfect.
We express circuit transformations as complex reflections 
with respect to the intersection form $\la\;,\;\ra$ in Theorem \ref{th:main}.
We give their representation matrices with respect to bases of 
$\lftH{\a}$ and $\tH{-\a}$ in Corollary \ref{cor:rep-mat}. We also give 
examples of representation matrices in the case of $m=3$, 
$\a_0,\a_1\in \Z$ and $\a_2+\a_3\in \Z$  in \S \ref{sec:example}.

We can define period matrices $\varPi_c^{lf}(\a,x)$ and 
$\varPi(\a,x)$ by the natural pairing between  $\lftH{\a}$ and 
$\ctC{\a}$, and that between $\tH{\a}$ and $\tC{\a}$, respectively, where 
$\tC{\a}$ and $\ctC{\a}$ are the twisted cohomology group and 
that with compact support (refer to (\ref{eq:twisted-cohomology}) for 
their definitions). 
%for their definition see \S \ref{sec:idetify}.
Under the condition (\ref{eq:old-condition}), each column vector of 
$\varPi_c^{lf}(\a,x)$ and $\varPi(\a,x)$ 
is a fundamental system of solutions to $\cF_D(a',b',c')$ for 
some $(a',b',c')$, of which difference from $(a,b,c)$ is an integral vector.
Under the condition (\ref{eq:non-int}), 
$\cM^\a$ and $\cN^\a$ can be regarded as the monodromy representations of 
$\varPi_c^{lf}(\a,x)$ and $\varPi(\a,x)$, 
%respectively,   
though they 
%the preiod matrices $\varPi_c^{lf}(\a,x)$ and $\varPi(\a,x)$ 
do not always include a fundamental system of solutions to $\cF_D(a,b,c)$. 
In general, the stalk of $\lfLS{\a}$ at $x$ cannot be regarded as 
the local solution space to $\cF_D(a,b,c)$ around $x$ 
under only the condition (\ref{eq:non-int}). 
To identify these spaces,  
% To identify the stalk of $\lfLS{\a}$ at $x$ with the local solution space 
% of $\cF_D(a,b,c)$ around $x$, 
%To regard $\cM^\a$ as the monodromy representation of $\cF_D(a,b,c)$, 
we suppose the condition 
\begin{equation}
\label{eq:our-condition}
%\exists i_0,i_{m+1},i_{m+2}\in \{0,1,\dots,m,m+1,m+2\}\quad 
%\textrm{s.t.}\quad 
\a_{i_m},  \a_{i_{m+1}},  \a_{i_{m+2}} \notin \Z, \
%(0\le i_m<i_{m+1}<i_{m+2}\le m+2),
\quad \a_{i_{m+1}}+\a_{i_{m+2}}\ne0, 
\end{equation}
and the non negative-integral condition (\ref{eq:non-negative}).
Under these conditions, the monodromy representation $\cM^\a$ can be regarded
as that of $\cF_D(a,b,c)$. 
Similarly, under the conditions (\ref{eq:our-condition}) and 
(\ref{eq:non-negative-0}), 
the monodromy representation $\cN^{-\a}$ can be regarded
as that of $\cF_D(-a,-b,-c)$. 

% By modifying a result in \cite{MY} similarly,   
% we study in \cite{MSTY} the monodromy representation of Appell's system 
% $\cF_2(a,b,c)$, which is valid even in a few reducible cases.  

\section{Lauricella's system $F_D$}
\label{sec:FD}
Lauricella's hypergeometric series $F_D(a,b,c;x)$ is defined by 
$$
F_D(a,b,c;x)
=\sum_{n\in \N^m}
\frac{(a,\sum_{i=1}^m n_i)\prod_{i=1}^m(b_i,n_i)}
{(c,\sum_{i=1}^m n_i)\prod_{i=1}^m(1,n_i)}
\prod_{i=1}^m x_i^{n_i},
$$
where  $x_1,\dots,x_m$ are complex variables with $|x_i|<1$ $(1\le i\le m)$, 
$a$, $b=(b_1,\dots,b_m)$ and $c$ are complex parameters,  
$c\notin -\N=\{0,-1,-2,\dots\}$, and 
$(b_i,n_i)=b_i(b_i+1)\cdots(b_i+ n_i-1)$.
It admits an Euler type integral representation:
\begin{equation}
\label{eq:intrep}
F_D(a,b,c;x)
=\frac{\G(c)} 
{\G(a)\G(c\!-\! a)}
\int_1^\infty u(t,x)\frac{dt}{t\!-\!1}, \quad 
u(t,x)=t^{\sum_i b_i-c}(t\!-\!1)^{c-a}\prod_{i=1}^m(t\!-\! x_i)^{-b_i},
\end{equation}
where the parameters $a$ and $c$ satisfy $\re(c)>\re(a)>0$.

The differential operators 
\begin{align}
\nonumber
x_i(1- x_i)\pa_i^2
+ (1- x_i)\sum\limits_{1\le j\le m}^{j\ne i}x_j\pa_i\pa_j 
+ [c- (a+ b_i+ 1)x_i]\pa_i 
- b_i\sum\limits_{1\le j\le m}^{j\ne i}x_j\pa_j- ab_i,&\\
\label{eq:LHGS}(1\le i\le m)& \\[2mm]
\nonumber
 (x_i-x_j)\pa_i\pa_j- b_j\pa_i+ b_i\pa_j,\hspace{3cm} (1\le i<j\le m)&
\end{align}
annihilate the series $F_D(a,b,c;x)$, where $\pa_i=\dfrac{\pa}{\pa x_i}$.
Lauricella's system $\cF_D(a,b,c)$ is defined by 
the ideal generated by these operators  
in the Weyl algebra $\C[x_1,\dots,x_m]\la \pa_1,\dots,\pa_m\ra$. 
Though the series $F_D(a,b,c;x)$ is not defined 
when $c\in -\N=\{0,-1,-2,\dots\}$, the system $\cF_D(a,b,c)$ 
is valid even in this case. 
It is a regular holonomic system of rank $m+1$ with singular locus 
\begin{equation}
\label{eq:sing-loc}
S=\big\{x\in \C^m\Big| \prod_{i=1}^m [x_i(1-x_i)]
\prod_{1\le i<j\le m} (x_i-x_j)=0\big\}\cup 
(\cup_{i=1}^\infty\{ x_i=\infty\})\subset (\P^1)^m.
\end{equation}
We set 
$$X=(\P^1)^m-S=\big\{(x_1,\dots,x_m)\in \C^m\mid \prod_{0\le i<j\le m+1}
(x_j-x_i)\ne 0\big\},$$
where $x_0=0$ and $x_{m+1}=1$. 
We introduce a notation 
$$\cx=(x_0,x_1,\dots,x_m,x_{m+1},x_{m+2})=(0,x_1,\dots,x_m,1,\infty)
=(0,x,1,\infty)$$
for $x\in X$.
Let $\Sol_x(a,b,c)$ be the vector space of solutions to $\cF_D(a,b,c)$ on 
a small simply connected neighborhood $U(\subset X)$ of $x$. 
It is called the local solution space to $\cF_D(a,b,c)$ around $x$, 
and it is $(m+1)$-dimensional. 
If the improper integral 
\begin{equation}
\label{eq:int-rep}
\int_{x_i}^{x_j} u(t,x)\frac{dt}{t\!-\!1}\quad (0\le i<j\le m+2)
\end{equation}
converges, then it gives an element of $\Sol_x(a,b,c)$.

We set 
$$
\fX=\Big\{(x,t)\in \C^m\times \C\;\Big|\; x\in X,\ 
%t(t-1)
\prod_{i=0}^{m+1} (t-x_i)\ne0
\big\},
$$
and  
$$
T_x=\Big\{t\in \C\;\Big|
%t(t-1)
\prod_{i=0}^{m+1} (t-x_i)\ne 0\Big\}=\C-\{0,x_1,\dots,x_m,1\}
$$
for any fixed $x\in X$.  
Note that $T_x$ is the preimage of $x$ under
the projection 
$$\pr:\fX\ni (x,t)\mapsto x\in X.$$

\section{Twisted homology groups}
\label{sec:homology}
In this section, we prepare facts about twisted homology groups associated 
with the Euler type integral (\ref{eq:intrep}) for our study.

Throughout this paper, we assume the conditions (\ref{eq:cond-wa}) 
and (\ref{eq:non-int}) on $\a$. We put 
$$\l=(\l_0,\l_1,\dots,\l_{m},\l_{m+1},\l_{m+2}), \qquad 
\l_i=\exp(2\pi\sqrt{-1}\a_i) \quad (0\le i\le m+2).$$
Note that $\prod_{i=0}^{m+2} \l_i=1$. By regarding $\l_i$ as indeterminants, 
we have  the  rational function field 
$\C(\l)=\C(\l_0,\l_1,\dots,\l_m,\l_{m+1})$. 
Let $\cL^\a$ be a locally constant sheaf on $\fX$ defined by a multi-valued 
function 
%$$u'(t,x)=t^{\a'_0}(t-x_1)^{\a'_1}\cdots
%(t-x_m)^{\a'_m}(t-x_{m+1})^{\a'_{m+1}},$$
$$u(t,x)=\prod_{i=0}^{m+1}(t-x_i)^{\a_i}=
t^{\a_0}(t-x_1)^{\a_1}\cdots (t-x_m)^{\a_m}(t-1)^{\a_{m+1}},$$
and $\cL_x^\a$ be that on $T_x$ defined by its restriction $u_x=u_x(t)$ 
to $T_x$.
For a fixed $x\in X$, we define a vector space 
$$\CC_k(u_x)=\Big\{\sum_\nu  w_\nu \cdot
(\tau_\nu,u_x|_{\tau_\nu})\mid w_\nu\in \C(\l)\Big\}$$
over $\C(\l)$, where the sum is finite, $\tau_\nu$ is a $k$-chain 
in $T_x$ and $u_x|_{\tau_\nu}$ is a branch of $u_x$ on $\tau_\nu$.
We define a twisted homology group as a quotient space 
\begin{equation}
\label{eq:t-H}
\tH{\a}
=\ker(\pa_u:\CC_1(u_x)\to \CC_2(u_x))/\pa_u(\CC_0(u_x)),
\end{equation}
where $\pa_u$ is a boundary operator defined by 
$$\pa_u(\tau_\nu,u_x|_{\tau_\nu})=(\pa\tau_\nu,
(u_x|_{\tau_\nu})|_{\pa\tau_\nu}).$$
Similarly we have a locally finite twisted homology group
\begin{equation}
\label{eq:lf-t-H}
\lftH{\a}
=\ker(\pa_u:\CC^{lf}_1(u_x)\to \CC^{lf}_2(u_x))/\pa_u(\CC^{lf}_0(u_x)),
\end{equation}
where $\CC^{lf}_k(u_x)$ is defined by extending finite sums to 
locally finite sums for $\CC_k(u_x)$.  
%the space of locally finite twisted $k$-chains in $T_x$.  

The dimensions of $\lftH{\a}$ and $\tH{\a}$ are equal to $-\chi(T_x)$ 
by \cite[Theorem 1]{C},  
where $\chi(T_x)$ is the Euler number of $T_x$. 
%Since $$\chi(T_x)=\chi(\C)-\chi(\{0,1,x_1,\dots,x_m\})=-(m+1),$$
Thus
we have the following.
\begin{fact}
\label{fact:dim}
% The twisted homology greoups $\tH{\a}$ and $\lftH{\a}$ are 
% $(m+1)$-dimensional. 
$$\dim\tH{\a}=\dim\lftH{\a}=m+1.$$
\end{fact}

We have a natural map 
\begin{equation}
\label{eq:natural-map}
\jmath_h^\a:\tH{\a}\to \lftH{\a}
\end{equation}
by regarding a finite sum as a locally finite sum.
This map is isomorphic under the condition (\ref{eq:old-condition}) and 
its inverse $\reg$ is called the regularization.
However, under our assumption (\ref{eq:non-int}), 
it does not hold in general.
\begin{lemma}
\label{lem:non-iso}
If there exists a parameter $\a_i\in \Z$ then  
the map $\jmath_h^\a$ is not isomorphic. 
%(\ref{eq:our-condition})
Even in this case, we have an isomorphism
$$\reg:\im(\jmath_h^\a) \to \tH{\a}/\ker(\jmath_h^\a).$$ 
\end{lemma}
\begin{proof}
Suppose that $\a_i$ is an integer. 
Let $\odot_i$ be an annulus 
$$\{t\in T_x\mid 0<|t-x_i|\le \e\},$$ 
and $\circlearrowleft_i$ be 
its boundary, where  $\e$ is a small positive real number.
For the case $i=m+2$, they are regarded as
$$\odot_{m+2}=\{t\in T_x\mid |t|\ge1/\e\}, \quad 
\circlearrowleft_{m+2}=\{t\in T_x\mid |t|=1/\e\}.$$ 
Since $u_x(t)$ is single-valued on $\odot_i$, we can regard 
$(\odot_i,u_x(t)|_{\odot_i})$ as an element of $\CC^{lf}_2(u_x)$.
Thus its image 
$(\circlearrowleft_i,u_x(t)|_{\circlearrowleft_i})$
under $\pa_u$ is $0$ as an element of $\lftH{\a}$.
However, we cannot regard $(\odot_i,u_x(t)|_{\odot_i})$ 
as an element of $\CC_2(u_x)$, since 
$\odot_i$ does not admit an expression as a finite union of $2$-simplexes.
In fact, we will show that 
$(\circlearrowleft_i,u_x(t)|_{\circlearrowleft_i})$ is not $0$ 
as an element of $\tH{\a}$ in Proposition \ref{prop:perfect}. 
It is elementary that 
$\tH{\a}/\ker(\jmath_h^\a)$ is isomorphic to $\im(\jmath_h^\a)$.
\end{proof}

We give  bases of $\lftH{\a}$ and $\tH{-\a}$.
Put
\begin{equation}
\label{eq:card-int}
r=\#\{0\le i\le m+2 \mid \a_i\in \Z\},
\end{equation}
and suppose that 
\begin{equation}
\label{eq:non-int-2-parameters}
\a_{i_0},\a_{i_1},\dots,\a_{i_{r-1}}\in \Z
,\quad \a_{i_{r}},\dots, \a_{i_{m+1}},\a_{i_{m+2}}\notin \Z,
\end{equation}
and the corresponding points $x_i$ $(i=0,\dots,m+2)$ are aligned
\begin{equation}
\label{eq:x-align}
x_{i_0}<x_{i_1}<\cdots<x_{i_{m}}<x_{i_{m+1}}<x_{i_{m+2}},
\end{equation}
where 
$$\left\{\begin{array}{cccc} 
i_{m+2}=m+2,& x_{i_{m+2}}=\infty &\textrm{if}& \a_{m+2}\notin\Z,\\
i_0=m+2,& x_{i_0}=-\infty &\textrm{if}& \a_{m+2}\in\Z.\\
\end{array}\right.
$$
Let $\imath$ be an element of the symmetric group $\frak{S}_{m+3}$ satisfying 
\begin{equation}
\label{eq:permutation}
\imath(i_0)=0, \ \imath(i_1)=1,\ \dots,\ \imath(i_m)=m,\ 
\imath(i_{m+1})=m+1,\ \imath(i_{m+2})=m+2.
\end{equation}
Then it satisfies $\imath^{-1}(p)=i_p$ for  $0\le p\le m+2$.
We fix these $x_1,\dots,x_m$.
% Recall that 
% $$\circlearrowleft_i=\{t\in T_x\mid |t-x_i|=\e\} \ 
% (0\le i\le m+1),\quad 
% \circlearrowleft_{m+2}=\{t\in T_x\mid |t|= 1/\e\}
% $$ 
% for a small real positive number $\e$. 
Suppose that the circle 
$$\circlearrowleft_i=\{t\in T_x\mid |t-x_i|=\e\} \quad  
(0\le i\le m+1),
$$
is positively oriented with terminal $\hat x_i=x_i+\sqrt{-1}\e$ 
for $0\le i\le m+1$, see Figure \ref{fig:cycles}, 
and 
$$\circlearrowleft_{m+2}=\{t\in T_x\mid |t|= 1/\e\}$$
is negatively oriented with terminal $\hat x_{m+2}=\sqrt{-1}/\e$.

\begin{figure}[htb]
\begin{center}
\includegraphics[width=9cm]{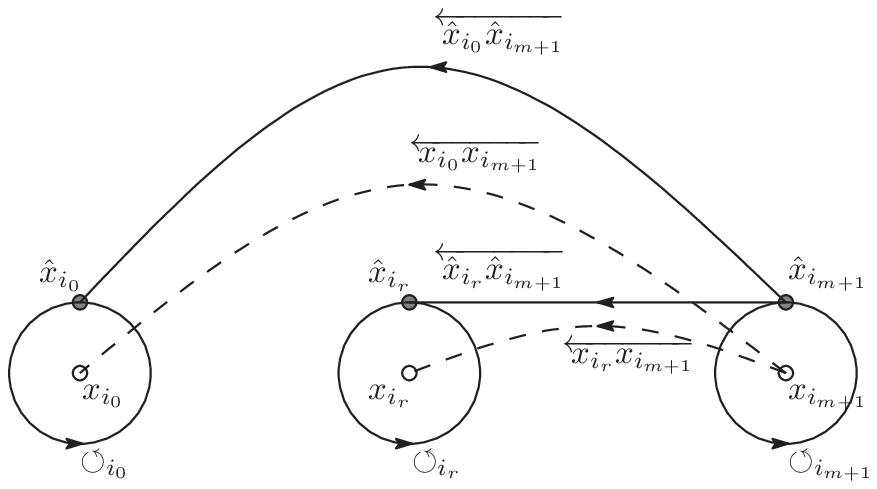}
\end{center}
\caption{Twisted cycles}
\label{fig:cycles}
\end{figure} 

We define twisted cycles by  
\begin{align}
\label{eq:lf-cycle}
\ell_k^\a&=(\overleftarrow{\;x_{i_k}{x_{i_{m+1}}}},u_x)\in \lftH{\a}
%-\frac{1}{1-\l_{i_{m+1}}}
%(\circlearrowleft_{i_{m+1}},u_x)
,\\
\label{eq:cycle}
\gamma_k^{-\a}&=
(\l^{-1}_{i_k}-1)(\overleftarrow{\; {\hat x_{i_k}}{\hat x_{i_{m+1}}}},u^{-1}_x)
-(\circlearrowleft_{i_{k}},u^{-1}_x)
+\frac{\l^{-1}_{i_{k}}-1}{\l^{-1}_{i_{m+1}}-1}
(\circlearrowleft_{i_{m+1}},u^{-1}_x)
\in \tH{-\a},
\end{align}
for $0\le k\le m$, where we take and fix a branch $u_x$ of $u_x(t)$ 
on the upper half space $\H\subset T_x$, and  
$\overleftarrow{\;x_{i_k}{x_{i_{m+1}}}}$ and  
$\overleftarrow{\; {\hat x_{i_k}}{\hat x_{i_{m+1}}}}$ are 
an oriented arc in $\H$ from $x_{i_{m+1}}$ to $x_{i_k}$ 
and that from $\hat x_{i_{m+1}}$  to 
$\hat x_{i_k}$, respectively, see Figure \ref{fig:cycles}.
Note that they are well-defined under 
the assumption (\ref{eq:non-int-2-parameters})
and that 
$$
\gamma_k^{-\a}=(\circlearrowleft_{i_{k}},u^{-1}_x) \quad (0\le k\le r-1)
$$
by $1-\l^{-1}_{i_k}=0$ for $0\le k\le r-1$.

\begin{definition}[\cite{AoKi},\cite{Y}]
\label{def:intersection}
The intersection form between 
$\lftH{\a}$ and $\tH{-\a}$ is defined by
\begin{equation}
\label{eq:intersection}
\la \ell^\a,\g^{-\a}\ra 
=\sum_{\mu,\nu}\sum_{t^\sigma_\tau\in \sigma_\mu\cap \tau_{\nu}}
(z_\mu\cdot w_{\nu})\cdot [\sigma_\mu,\tau_\nu]_{t^\mu_\nu}
\cdot 
(u_x(t^\mu_\nu)|_{\sigma_\mu}\cdot u_x^{-1}(t^\mu_\nu)|_{\tau_{\nu}})
\in \C(\l)
,
\end{equation}
where 
$$
\ell^\a=\sum_{\mu}  z_{\mu}\cdot (\sigma_{\mu},u_x|_{\sigma_{\mu}})
\in \lftH{\a},\quad 
\g^{-\a}=\sum_{\nu}  w_{\nu} \cdot (\tau_{\nu},u_x^{-1}|_{\tau_{\nu}})
\in \tH{-\a},$$
the formal sum for $\mu$ is locally finite, 
the formal sum for $\nu$ is finite, and 
$1$-chains $\sigma_\mu$ and $\tau_\nu$ intersect transversally at most one 
point $t^\mu_\nu$  with the topological intersection number 
$[\sigma_\mu, \tau_{\nu}]_{t^\mu_\nu}=\pm1$. 
\end{definition}

\begin{proposition}
\label{prop:perfect}
Under the condition (\ref{eq:non-int}), the intersection form
$\la\;, \;\ra$ is perfect, and  
the twisted cycles $\ell_0^\a,\ell_1^\a,\dots,\ell_m^\a$ and 
$\g_0^{-\a},\g_1^{-\a},\dots,\g_m^{-\a}$ are bases of 
$\lftH{\a}$ and $\tH{-\a}$, respectively.
\end{proposition}
\begin{proof}
We compute the intersection matrix 
$H=\big(\la \ell_k^\a,\g_{k'}^{-\a}\ra\big)_{0\le k,k'\le m}$. 
The locally finite chain $\overleftarrow{\;x_{i_0}x_{i_{m+1}}}$ and 
finite chains consisting of $\g_0^{-\a}$ intersect at two points 
$t_{0}^{0}$ and $t_{m+1}^{0}$ on $\circlearrowleft_{i_{0}}$ and 
on $\circlearrowleft_{i_{m}}$, respectively.
The topological intersection numbers are 
$$[\overleftarrow{\;x_{i_0}x_{i_{m+1}}}, \circlearrowleft_{i_{0}}]_{t_{0}^{0}}=
-1,\quad 
[\overleftarrow{\;x_{i_0}x_{i_{m+1}}},\circlearrowleft_{i_{m+1}}]_{t_{m+1}^{0}}
=1,$$
and the products of branches of $u_x(t)$ and $1/u_x(t)$ at 
$t_{0}^{0}$ and $t_{m+1}^{0}$ are 
$$
u_x(t_{0}^{0})\cdot \frac{1}{u_x(t_{0}^{0})}=\l_{i_0}^{-1},\quad 
u_x(t_{m+1}^{0})\cdot \frac{1}{u_x(t_{m+1}^{0})}=1.
$$
By considering coefficients, we have 
$$
\la \ell_0^\a,\g_0^{-\a}\ra=
%(-1)\cdot (-1)\cdot 
\l_{i_0}^{-1}
+\frac{\l_{i_{0}}^{-1}-1}{\l_{i_{m+1}}^{-1}-1}
%\cdot(+1)\cdot 1
= \frac{\l_{i_0}\l_{i_{m+1}}-1}{\l_{i_0}(\l_{i_{m+1}}-1)}
=1+\frac{\l_{i_0}-1}{\l_{i_0}(\l_{i_{m+1}}-1)}.
$$
Similarly, we have 
$$
\la \ell_0^\a,\g_k^{-\a}\ra=\frac{\l_{i_k}-1}{\l_{i_k}(\l_{i_{m+1}}-1)},
\quad 
\la \ell_k^\a,\g_0^{-\a}\ra=\frac{(\l_{i_0}-1)\l_{i_{m+1}}}
{\l_{i_0}(\l_{i_{m+1}}-1)}
$$
for $1\le k\le m$.
Hence we have 
\begin{equation}
\label{eq:int-mat}
H=I_{m+1}+
\frac{1}{\l_{i_{m+1}}-1}
\begin{pmatrix} 
1& 1& \cdots & 1\\
\l_{i_{m+1}}& 1 & \cdots &1 \\
\vdots &\ddots &\ddots &\vdots \\
\l_{i_{m+1}}& \cdots & \l_{i_{m+1}} &1 
 \end{pmatrix}
\diag\big(
\overbrace{0,\dots,0}^r,
\frac{\l_{i_r}-1}{\l_{i_r}},
\dots,
\frac{\l_{i_m}-1}{\l_{i_m}}
\big),
\end{equation}
where $I_{m+1}$ is the unit matrix of size $m+1$, and 
$\diag(z_0,\dots,z_m)$ denotes the diagonal matrix with diagonal entries 
$z_0,\dots,z_m$. 
% \begin{align*}
% \frac{-1}{1-\l_{i_{m+1}}^{-1}}
% \begin{pmatrix} 
% 1& 1& \cdots & 1\\
% \l_{i_{m+1}}^{-1}& 1 & \cdots &1 \\
% \vdots &\ddots &\ddots &\vdots \\
% \l_{i_{m+1}}^{-1}& \cdots & \l_{i_{m+1}}^{-1} &1 
% \end{pmatrix}
% &\diag(
% %1-\l_{i_0},1-\l_{i_1},
% \overbrace{0,\dots,0}^r,1-\l_{i_r}^{-1}\dots,1-\l_{i_m}^{-1})\\
% % \begin{pmatrix}
% % 1-\l_{i_0}& & & \\
% % &1-\l_{i_1}& & \\
% % & &\ddots& \\
% % & & &1-\l_{i_m}
% % \end{pmatrix}
% -&
% \diag(\overbrace{1,\dots,1}^r,\l_{i_r}^{-1}\dots,\l_{i_m}^{-1}).
% \end{align*}
Since its determinant is 
$$\frac{1-\l_{i_{m+2}}}{1-\l_{i_{m+1}}^{-1}}\ne0,$$ 
%-(1-mu5)/(mu5*(1-mu4))
$\ell_0^{\a},\dots,\ell_m^{\a}$ and $\g_{0}^{-\a},\dots,\g_{m}^{-\a}$ 
are  bases of $\lftH{\a}$ and $\tH{-\a}$, and the 
intersection form is perfect.
\end{proof}

\begin{remark}
\label{rem:extended-cycles}
We extend the notations in (\ref{eq:lf-cycle}) and (\ref{eq:cycle}) 
to $k=m+1,m+2$. 
It is obvious that $\ell^\a_{m+1}=0$, $\g^{-\a}_{m+1}=0$.
Since 
$$\la \ell^\a_{m+2},\g^{-\a}_{k}\ra=\frac{\l_{i_k}^{-1}-1}
{\l_{i_{m+1}}^{-1}-1}\cdot \l_{i_{m+1}}^{-1},\quad 
\la \ell^\a_{k},\g^{-\a}_{m+1}\ra=\frac{\l_{i_{m+2}}^{-1}-1}
{\l_{i_{m+1}}^{-1}-1},
$$
the cycles $\ell^\a_{m+2}$ and $\g^{-\a}_{m+2}$ are expressed 
as linear combinations
\begin{align}
\nonumber
&\ell^\a_{m+2}=\ex_{m+2} \cdot
\begin{pmatrix}
\ell^\a_0\\
\ell^\a_1\\
\vdots\\
\ell^\a_m
\end{pmatrix}, 
\quad \g^{-\a}_{m+2}=
\big(\g^{-\a}_{0},\g^{-\a}_{1},\dots,\g^{-\a}_{m}\big)\cdot 
\ex_{m+2}^*,\\
\label{eq:e(m+2)}
&\ex_{m+2}
=\frac{-\l_{i_{m+2}}}{\l_{i_{m+2}}-1}
\Big((\l_{i_0}-1), \l_{i_0}(\l_{i_1}-1),
\cdots,
\l_{i_0}\l_{i_1}\cdots \l_{i_{m-1}}(\l_{i_m}-1)\Big)
,\\
\label{eq:e*(m+2)}
&\ex_{m+2}^*=\left(
\begin{array}{r}
\l_{i_0}\l_{i_1}\cdots \l_{i_{m}}\\
\l_{i_1}\cdots \l_{i_{m}}\\
\vdots\quad\\
\l_{i_{m}}
\end{array}\right).
\end{align}
% \begin{align*}
% \ell^\a_{m+2}&=\frac{-\l_{i_{m+2}}}{\l_{i_{m+2}}-1}
% \Big((\l_{i_0}-1), \l_{i_0}(\l_{i_1}-1),
% \cdots,
% \l_{i_0}\l_{i_1}\cdots \l_{i_{m-1}}(\l_{i_m}-1)\Big)
% \cdot 
% \begin{pmatrix}
% \ell^\a_0\\
% \ell^\a_1\\
% \vdots\\
% \ell^\a_m
% \end{pmatrix}
% ,\\
% %
% \g^{-\a}_{m+2}&=
% \big(\g^{-\a}_{0},\g^{-\a}_{1},\dots,\g^{-\a}_{m}\big)
% \cdot (-\l_{i_{m+1}})\left(
% \begin{array}{r}
% \l_{i_0}\l_{i_1}\cdots \l_{i_{m}}\\
% \l_{i_1}\cdots \l_{i_{m}}\\
% \vdots\quad\\
% \l_{i_{m}}
% \end{array}\right).
% \end{align*}
Their intersection number is 
$$
\la \ell^\a_{m+2},\g^{-\a}_{m+2}\ra =
\frac{\l_{i_{m+1}}\l_{i_{m+2}}-1}
{(\l_{i_{m+1}}-1)\l_{i_{m+2}}}
=
1+\frac{\l_{i_{m+2}}-1}
{(\l_{i_{m+1}}-1)\l_{i_{m+2}}}
.
$$
\end{remark}

\section{Local systems}
\label{sec:loc-system}
We take a base point $\dot x\in X$ so that 
$$(\dot x_0,\dot x_1,\dots,\dot x_m,\dot x_{m+1}, \dot x_{m+2}) 
= (0,\dot x_1,\dots,\dot x_m,1, \infty) 
$$
are aligned as in 
(\ref{eq:x-align}) for a fixed parameter $\a$ satisfying 
(\ref{eq:non-int-2-parameters}).
We set 
$$\C_p(\dot x)=\{
(\dot x_1,\dots, \dot x_{p-1},x_p, \dot x_{p+1},\dots, \dot x_m)
\mid x_p\in \C\}\quad (1\le p\le m),$$
which are lines in $\C^m$ passing through $\dot x$. 
For distinct indexes $1\le p\le  m$ and $0\le q\le m+1$, 
let  $\rho_{pq}$ be a loop 
in $X\cap \C_p(\dot x)$ starting from $x_p=\dot x_p$, approaching to 
$\dot x_q$ via the upper half space in $\C_p(\dot x)$, 
turning $\dot x_q$ once positively, and tracing back to $\dot x_p$.

\begin{fact}
\label{fact:pi1}
The fundamental group $\pi_1(X,\dot x)$ is generated by the loops 
$\rho_{pq}$, 
where $0\le p<q\le m+1$, $(p,q)\ne(0,m+1)$, and 
$\rho_{0p}$ is regarded as the loop $\rho_{p0}$ in $X\cap \C_p(\dot x)$.
\end{fact}

By the local triviality of the spaces 
$\lftH{\a}$,  $\tH{\a}$,  $\tH{-\a}$,  $\lftH{-\a}$,  
we have the local systems 
$$
\lfLS{\a}=\bigcup_{x\in X}\lftH{\a}, \quad 
\LS{\a}=\bigcup_{x\in X}\tH{\a}, 
$$

$$
\LS{-\a}=\bigcup_{x\in X}\tH{-\a}, \quad 
\lfLS{-\a}=\bigcup_{x\in X}\lftH{-\a}, 
$$
over $X$. 
% We denote the stalks of these local system over $x\in X$ 
% by the same symbols 
% $\lftH{\a}$,  $\tH{\a}$,  $\tH{-\a}$,  $\lftH{\a}$, respectively. 

\begin{proposition}
\label{prop:stable}
\begin{enumerate}
\item 
\label{item:commute}  
The natural map 
%(\ref{eq:natural-map}) 
$\jmath_h^\a:\tH{\a}\to \lftH{\a}$
commutes with horizontal deformations  in $\lfLS{\a}$.
\item 
\label{item:stable}  
The intersection form $\la \;,\; \ra$ between 
$\lftH{\a}$ and  $\tH{-\a}$ is stable under horizontal deformations  
in $\lfLS{\a}$ and $\LS{-\a}$.
\end{enumerate}
\end{proposition}

\begin{proof}
(\ref{item:commute}) 
It is obvious by the definition of $\jmath_h^\a$.

\noindent(\ref{item:stable})
Note that the intersection matrix $H$ 
%for bases of $\lftH{\a}$ and  $\tH{-\a}$ 
in (\ref{eq:int-mat}) 
is independent of $x\in X$. 
\end{proof}

% We denote the stalks of these local system over $x\in X$ 
% by the same symbols 
% $\lftH{\a}$,  $\tH{\a}$,  $\tH{-\a}$,  $\lftH{-\a}$, respectively. 

\section{Monodromy representations}
\label{sec:monod-rep}

Let $U$ be a small simply connected neighborhood of $\dot x$ 
contained in $X$.  We set four trivial vector bundles 
$$
\begin{array}{ll}
\lfSV{\a}=\bigcup\limits_{x\in U} \lftH{\a}\subset \lfLS{\a}, &
\SV{\a}=\bigcup\limits_{x\in U} \tH{\a}\subset \LS{\a},\\
\SV{-\a}=\bigcup\limits_{x\in U} \tH{-\a}\subset \LS{-\a},&
\lfSV{-\a}=\bigcup\limits_{x\in U} \lftH{-\a}\subset \lfLS{-\a}.
\end{array}
$$
% We can identify these spaces with the stalks of 
% $\lfLS{\a}$, $\LS{\a}$, $\LS{-\a}$, and $\lfLS{-\a}$, 
% % $\lftH{\a}$,  $\tH{\a}$,  $\tH{-\a}$,  $\lftH{-\a}$
% over $x\in U$, respectively.  
% In this section, 
We identify sections of $\lfSV{\a}$, $\SV{\a}$, 
$\SV{-\a}$ and $\lfSV{-\a}$ with  elements of  
$\lftH{\a}$,  $\tH{\a}$,  $\tH{-\a}$ and  $\lftH{-\a}$ for a fixed 
element $x\in U$.

A loop $\rho$ with terminal $\dot x$ in $X$
induces $\C$-linear isomorphisms
\begin{align*}
\cM^\a_\rho&:\lfSV\a \ni \ell^\a \mapsto \rho_*\ell^\a \in \lfSV\a,\\
\cN^\a_\rho&:\SV\a \ni \g^\a \mapsto \rho_*\g^\a \in \SV\a,\\
\cN^{-\a}_\rho&:\SV{-\a} \ni \g^{-\a} \mapsto \rho_*\g^{-\a} \in \SV{-\a},\\
\cM^{-\a}_\rho&:\lfSV{-\a} \ni \ell^{-\a} \mapsto 
\rho_*\ell^{-\a} \in \lfSV{-\a},
\end{align*}
where 
$\ell^\a\in \lftH{\a}$, $\g^\a \in \tH{\a}$,  $\g^{-\a}\in \tH{-\a}$ and 
$\ell^{-\a}\in \lftH{-\a}$ are regarded as sections of  
%$\dot\ell^\a$, $\dot\g^\a$  $\dot\g^{-\a}$, $\dot\ell^{-\a}$ 
%in $\dlftH{\a}$,  $\dtH{\a}$,  $\dtH{-\a}$, $\dlftH{-\a}$, respectively,
$\lfSV\a$, $\SV\a$, $\SV{-\a}$ and $\lfSV{-\a}$, 
and $\rho_*\ell^\a$, $\rho_*\g^\a$  $\rho_*\g^{-\a}$, $\rho_*\ell^{-\a}$
are their continuations along the loop $\rho$.
They are called circuit transformations along $\rho$.
The map $\rho \mapsto \cM^\a_\rho$
induces a homomorphism 
$$\cM^\a:\pi_1(X,\dot x)\ni \rho \mapsto \cM^\a_\rho\in GL( \lfSV\a),$$ 
which is called the monodromy representation.  
Similarly, we have the monodromy representations   
\begin{align*}
\cN^\a&:\pi_1(X,\dot x) \to GL(\SV\a),\\
\cN^{-\a}&:\pi_1(X,\dot x) \to GL(\SV{-\a}),\\
\cM^{-\a}&:\pi_1(X,\dot x) \to GL(\lfSV{-\a}).
\end{align*}

\begin{proposition}
\label{prop:reduce}
Suppose that the condition (\ref{eq:non-int}). If there exists a 
parameter $\a_i\in\Z$ then the monodromy representations 
$\cM^\a$, $\cN^\a$, $\cN^{-\a}$ and $\cM^{-\a}$ are reducible.
Their proper invariant subspaces are $\mathrm{im}(\jmath_h^\a)$, 
$\ker(\jmath_h^\a)$, $\ker(\jmath_h^{-\a})$ and  
$\mathrm{im}(\jmath_h^{-\a})$, respectively,
where the natural map $\jmath_h^\a$ is given in (\ref{eq:natural-map}).
\end{proposition}

\begin{proof}
By Lemma \ref{lem:non-iso}, the image $\mathrm{im}(\jmath_h^\a)$ and 
the kernel $\ker(\jmath_h^\a)$ of $\jmath_h^\a$ 
%$:\tH{\a}\to\lftH{\a}$ 
are proper subspaces.
By Proposition \ref{prop:stable} (\ref{item:commute}), 
they are invariant under any circuit transformations.
Hence $\cM^\a$ and  $\cN^\a$ are reducible. For the reducibility of 
$\cN^{-\a}$ and $\cM^{-\a}$, use the sign change $\a\mapsto -\a$. 
\end{proof}

Since $\cM^{-\a}$ and $\cN^{\a}$ are obtained from the sign change
$\a \mapsto -\a$ for $\cM^{\a}$ and $\cN^{-\a}$, 
we mainly treat $\cM^{\a}$ and $\cN^{-\a}$. 

\begin{proposition}
\label{prop:keep-int}
\begin{enumerate}
\item 
\label{item:invariant}
The intersection form is invariant under the monodromy representations,
that is 
$$
\la \cM_\rho^\a(\ell^\a),\cN_\rho^{-\a}(\g^{-\a})\ra
=\la \ell^\a,\g^{-\a}\ra,
$$
where $\rho$ is a loop in $X$ with terminal $\dot x$, and 
$\ell^\a$ and $\g^{-\a}$ are sections in $\lfSV{\a}$ and $\SV{-\a}$. 
\item 
\label{item:dual}
If $\b$ is an eigenvalue of $\cM_\rho^\a$, then 
$\b^{-1}$ is an eigenvalue of $\cN_\rho^{-\a}$. 
\item 
\label{item:orth}
Let $\ell^\a$ be a $\b$-eigenvector of $\cM_\rho^\a$, 
and $\g^{-\a}$ be a $\b'$-eigenvector of $\cN_\rho^{-\a}$.
If $\b\b'\ne 1$ then $\la \ell^\a,\g^{-\a}\ra=0$.
If $\la \ell^\a,\g^{-\a}\ra\ne0$ then $\b\b'= 1$.

\end{enumerate}
\end{proposition}
\begin{proof} (\ref{item:invariant}) 
This property is a natural consequence from Proposition 
\ref{prop:stable} (\ref{item:stable}). 

\noindent(\ref{item:dual})
Since $\cM_\rho^\a$ is invertible, $\b$ is different from $0$. 
Let $\ell^\a$ be a $\b$-eigenvector of $\cM_\rho^\a$ and  
let $\ell_{\rho,0}^\a(=\ell^\a),\ell_{\rho,1}^\a,\dots,\ell_{\rho,m}^\a$ be 
a basis of $\lfSV{\a}$. 
Then there exists the dual basis 
$\g_{\rho,0}^{-\a},\g_{\rho,1}^{-\a},\dots,\g_{\rho,m}^{-\a}$ of 
$\SV{-\a}$ with respect to 
the intersection form by Proposition \ref{prop:perfect}. 
Let $M_\rho^\a$ be the representation matrix of $\cM_\rho^{-\a}$ 
with respect to $\tr(\ell_{\rho,0}^{-\a},\ell_{\rho,1}^{-\a},\dots,
\ell_{\rho,m}^{-\a})$ and 
let $N_\rho^{-\a}$ be that of $\cN_\rho^{-\a}$ 
with respect to $(\g_{\rho,0}^{-\a},\g_{\rho,1}^{-\a},\dots,
\g_{\rho,m}^{-\a})$. 
By (\ref{item:invariant}), we have
$M_\rho^\a N_\rho^{-\a}=I_{m+1}$.
Hence $\b^{-1}$ is an eigenvalue of $\cN_\rho^{-\a}$ and $\g_{\rho,0}^{-\a}$ 
is $\b^{-1}$-eigenvector of $\cN_\rho^{-\a}$. 

\noindent(\ref{item:orth}) 
Since 
$$\la \ell^\a,\g^{-\a}\ra=\la \cM_\rho ^\a(\ell^\a),\cN_\rho^{-\a}(\g^{-\a})\ra
=\b\b'\la \ell^\a,\g^{-\a}\ra,$$
we have $(1-\b\b')\la \ell^\a,\g^{-\a}\ra=0$. 
\end{proof}

To characterize the monodromy representations $\cM^{\a}$ and $\cN^{-\a}$, 
it is sufficient to express the circuit transformations 
$$\cM^\a_{pq}=\cM^\a_{\rho_{pq}},\quad 
\cN^{-\a}_{pq}=\cN^{-\a}_{\rho_{pq}}\quad 
(0\le p<q\le m+1,\ (p,q)\ne(0,m+1))
$$
by Fact \ref{fact:pi1}.

%where $0\le p<q\le m+1,$ $(p,q)\ne(0,m+1).$
Under the condition (\ref{eq:old-condition}), 
$\cN^{-\a}_{pq}$ is obtained from the sign change $\a\mapsto -\a$ for 
$\cM^\a_{pq}$, since the map $\jmath_h^\a$ is isomorphic. 
Moreover, the circuit transform $\cM^\a_{pq}$ is expressed 
by the intersection form $\la\ ,\ \ra$ as follows.

\begin{fact}[{\cite[Theorem 5.1]{M2}}] 
\label{fact:week-monodromy} 
Under the assumption (\ref{eq:old-condition}), 
$\cM^\a_{pq}$ is expressed as
$$\cM^\a_{pq}: \ell^\a \mapsto  
\ell^\a-(\l_p-1)(\l_q-1)\la  \ell^\a , 
\reg(\ell_{\imath(p)\imath(q)}^{-\a})
\ra \ell_{\imath(p)\imath(q)}^{\a},$$
where $\imath\in \frak{S}_{m+3}$ 
is given in (\ref{eq:permutation}),  
$$\ell_{\imath(p)\imath(q)}^{\a}=
\ell_{\imath(q)}^{\a}-\ell_{\imath(p)}^{\a},\quad 
\ell_{\imath(p)\imath(q)}^{-\a}=
\ell_{\imath(q)}^{-\a}-\ell_{\imath(p)}^{-\a},$$
$\ell^\a_i$ $(0\le i\le m+2)$ 
are defined in (\ref{eq:lf-cycle}) and Remark \ref{rem:extended-cycles}, 
$\ell_{i}^{-\a}$ is obtained from the sign change $\a\mapsto -\a$ 
for $\ell_{i}^{\a}$, 
and $\reg$ is the inverse of $\jmath_h^{-\a}:\tH{-\a}\to \lftH{-\a}$.
 \end{fact}

% It is shown in \cite[Ch.IV,\S7]{Y} that 
% $$\la \ell_{\imath(p)\imath(q)}^{\a}, 
% \reg(\ell_{\imath(p)\imath(q)}^{-\a})\ra
% =\frac{-(\l_p\l_q-1)}{(\l_p-1)(\l_q-1)}.
% $$
% Thus 
% %if $\l_p\l_q\ne 1$ then 
% $\ell_{\imath(p)\imath(q)}^{\a}$ 
% is a $\l_p\l_q$-eigenvector of $\cM^\a_{pq}$. 
% By Proposition \ref{prop:perfect}, the subspace 
% $$\reg(\ell_{\imath(p)\imath(q)}^{-\a})^\perp=
% \{\ell^\a\in \lftH{\a}\mid 
% \la \ell^\a,\reg(\ell_{\imath(p)\imath(q)}^{-\a})\ra
% =0\}
% $$
% is $m$-dimensional. It is easy to see that this is the eigenspace of 
% $\cM^\a_{pq}$ of eigenvalue $1$.
% We remark that if $\l_p\l_q=1$ then 
% $\ell_{\imath(p)\imath(q)}^{\a}$ belongs to 
% $\reg(\ell_{\imath(p)\imath(q)}^{-\a})^\perp$ and 
% $\cM^\a_{pq}$ is not diagonalizable.

We can modify this fact so that it is valid under the condition 
(\ref{eq:non-int}). 

\begin{theorem}
\label{th:main}
Suppose that the condition  (\ref{eq:non-int}) and   
the indexes $p$ and $q$ satisfy $0\le p<q\le m+1,$ $(p,q)\ne(0,m+1)$. 
% If $\a_p,\a_q\in \Z$ then 
% $\cM^\a_{pq}$ and $\cN^{-\a}_{pq}$ are identical. 
% If either $\a_p\notin \Z$ or $\a_q\notin\Z$ then 
The circuit transformations 
$\cM^\a_{pq}$ and $\cN^{-\a}_{pq}$ are expressed as
\begin{align}
\label{eq:cM-exp}
\cM^\a_{pq}:\hspace{3mm}  \ell^\a &\mapsto  
\ell^\a-\l_{p}\l_{q}
\la \ell^\a ,\g_{\imath(p)\imath(q)}^{-\a}\ra
\ell_{\imath(p)\imath(q)}^{\a},\\
\label{eq:cN-exp}
\cN^{-\a}_{pq}: \g^{-\a} &\mapsto  
\g^{-\a}+
\la \ell_{\imath(p)\imath(q)}^\a , \g^{-\a}\ra 
\g_{\imath(p)\imath(q)}^{-\a},
\end{align}
where  $\imath\in \frak{S}_{m+3}$ 
is given in (\ref{eq:permutation}),  
$$\ell_{\imath(p)\imath(q)}^{\a}=
\ell_{\imath(q)}^{\a}-\ell_{\imath(p)}^{\a},\quad 
\g_{\imath(p)\imath(q)}^{-\a}
=(\l_p^{-1}-1)\g_{\imath(q)}^{-\a}-(\l_q^{-1}-1)\g_{\imath(p)}^{-\a},$$
$\ell^\a_i$ and $\g_{i}^{-\a}$ $(0\le i\le m+2)$ 
are defined in (\ref{eq:lf-cycle}), (\ref{eq:cycle})
and Remark \ref{rem:extended-cycles}. 
\end{theorem}

\begin{remark}
\label{rem:cancel}
We set 
\begin{align*}
(\g_{\imath(p)\imath(q)}^{-\a})^\perp&=
\{\ell^\a\in \lftH{\a}\mid \la \ell^\a,\g_{\imath(p)\imath(q)}^{-\a}\ra=0\},\\
(\ell_{\imath(p)\imath(q)}^\a)^\perp&=
\{\g^{-\a}\in \tH{-\a}\mid \la \ell_{\imath(p)\imath(q)}^\a ,\g^{-\a}\ra=0\}.
\end{align*}
It is easy to see that these spaces belong to 
the $1$-eigenspace of the expressions (\ref{eq:cM-exp}) and 
(\ref{eq:cN-exp}), respectively, and that their dimensions are more then 
or equal to $m$.
If $\g_{\imath(p)\imath(q)}^{-\a}$ (resp. $\ell_{\imath(p)\imath(q)}^\a$) 
is the zero element, then $(\g_{\imath(p)\imath(q)}^{-\a})^\perp$ 
(resp. $(\ell_{\imath(p)\imath(q)}^\a)^\perp $) is $(m+1)$-dimensional  
and (\ref{eq:cM-exp}) (resp. (\ref{eq:cN-exp})) becomes the identity. 
If $\g_{\imath(p)\imath(q)}^{-\a}$ (resp. $\ell_{\imath(p)\imath(q)}^\a$) 
is different from the zero element  then 
$ (\g_{\imath(p)\imath(q)}^{-\a})^\perp$ 
(resp. $(\ell_{\imath(p)\imath(q)}^\a)^\perp$) is $m$-dimensional,
since the intersection form $\la\ ,\ \ra$ is perfect. 
In this case,   (\ref{eq:cM-exp}) (resp. (\ref{eq:cN-exp}) ) 
is characterized by the space $(\g_{\imath(p)\imath(q)}^{-\a})^\perp$ 
(resp. $(\ell_{\imath(p)\imath(q)}^\a)^\perp$)
and the image of an element in its complement. 
In particular, if $\a_p+\a_q\notin \Z$ then  we have 
\begin{equation}
\label{eq:intno-DV}
\la \ell_{\imath(p)\imath(q)}^\a ,\g_{\imath(p)\imath(q)}^{-\a}\ra=
\frac{1-\l_p\l_q}{\l_p\l_q}\ne0, 
\end{equation}
which implies that neither $\ell_{\imath(p)\imath(q)}^{\a}$ nor 
$\g_{\imath(p)\imath(q)}^{-\a}$ is the zero element.   
Moreover, we can rewrite 
(\ref{eq:cM-exp}) and (\ref{eq:cN-exp}) into  
complex reflections with respect to the intersection form $\la\ ,\ \ra$:
\begin{align*}
%\cM^\a_{pq}(\ell^\a)&=
\ell^\a&\mapsto 
\ell^\a-(1-\l_p\l_q)\frac{\la \ell^\a , \g_{\imath(p)\imath(q)}^{-\a}\ra }
{\la \ell_{\imath(p)\imath(q)}^\a ,\g_{\imath(p)\imath(q)}^{-\a}\ra }
\ell_{\imath(p)\imath(q)}^{\a},\\
%
%\cN^{-\a}_{pq}(\g^{-\a})&=
\g^{-\a}&\mapsto 
\g^{-\a}-(1-\l_p^{-1}\l_q^{-1})
\frac{\la \ell_{\imath(p)\imath(q)}^\a , \g^{-\a}\ra }
{\la \ell_{\imath(p)\imath(q)}^\a ,\g_{\imath(p)\imath(q)}^{-\a}\ra }
\g_{\imath(p)\imath(q)}^{-\a}.
\end{align*}
Note that 
$\ell_{\imath(p)\imath(q)}^{\a}(\notin (\g_{\imath(p)\imath(q)}^{-\a})^\perp)$ 
is a $\l_p\l_q$-eigenvector of (\ref{eq:cM-exp}), and that
$\g_{\imath(p)\imath(q)}^{-\a}(\notin (\ell_{\imath(p)\imath(q)}^{\a})^\perp)$ 
is a $\l_p^{-1}\l_q^{-1}$-eigenvector of (\ref{eq:cN-exp}).
\end{remark}
\begin{proof}[Proof of Theorem \ref{th:main}] $ $\\ \noindent
By tracing deformations of some cycles along $\rho_{pq}$, 
we study eigenspaces of $\cM^\a_{pq}$ and $\cN^{-\a}_{pq}$.  
By Proposition \ref{prop:keep-int} (\ref{item:dual}),  
it is sufficient to consider either $\cM^\a_{pq}$ or $\cN^{-\a}_{pq}$.
We can show that 
\begin{equation}
\label{eq:non-1-EV}
\cM^\a_{pq}(\ell_{\imath(p)\imath(q)}^{\a})
=(\l_p\l_q)\cdot \ell_{\imath(p)\imath(q)}^{\a}
\end{equation}
% is an eigenvector of $\cM^\a_{pq}$ of eigenvalue $\l_p\l_q$ 
by using the proof of \cite[Lemma 5.1]{M2}, 
which is valid under the condition (\ref{eq:non-int}). 
If $\ell_{\imath(p)\imath(q)}^{\a}$ does not degenerate then 
it becomes a $\l_p\l_q$-eigenvector of $\cM^\a_{pq}$.

We consider $m$ elements $\breve{\ell}_{\infty,j}^\a\in \lftH{\a}$
given by oriented arcs $\breve{\ell}_{\infty,j}$ in the lower half space 
from $x_{m+2}=\infty$ to $x_{i_j}$ $(i_j\ne p,q,m+2)$ 
and branches of $u_x$ on them. 
Here note that $p$ and $q$ are different from $m+2$ by our setting of 
$\rho_{pq}$.  
Since these arcs are not involved the deformation along the loop $\rho_{pq}$, 
if they do not degenerate then they become $1$-eigenvectors of $\cM^\a_{pq}$.  
It is easy to see that they belong to $(\g_{\imath(p)\imath(q)}^{-\a})^\perp$.
Let $E_{pq}^\a(1)$ be the subspace of $\lftH{\a}$ spanned by them.
We claim that 
%there are at least $(m-1)$ linearly independent elements of among them 
$$\dim E_{pq}^\a(1)\ge m-1$$
under the condition (\ref{eq:non-int}).  
In fact, we have at least $(m-1)$ elements $\breve{\ell}_{\infty,j}^\a$
$(j\ne \imath(p),\imath(q),m+1,\imath(m+2))$. 
(If $\imath(p)=m+1$ or $\imath(q)=m+1$ then we have $m$ elements.)
Note that the arc $\breve{\ell}_{\infty,j}$ intersects 
only one oriented circle $\circlearrowleft_{i_{j}}$ 
among chains defining $\g_k^{-\a}$ $(k\in\{0,1,\dots,m,m+2\}-
\{\imath(m+2)\})$.
Thus the intersection matrix 
$$H_{m-1}'=\big(\la \breve{\ell}_{\infty,j}^\a, 
\g_k^{-\a}\ra\big)_{j,k}
\quad (j,k\in \{0,1,\dots,m,m+2\} -\{\imath(p),\imath(q),\imath(m+2)\})$$ 
becomes a diagonal matrix of size $(m-1)$ with non-zero diagonal entries,
which means that they are linearly independent. 

To investigate the detailed structure of these eigenspaces,  
we need the following case studies on parameters. 
% Otherwise, it becomes a lower triangle matrix of size $m$ with non-zero 
% diagonal entries, since $\la \breve{\ell}_{\infty,m+1}^\a, \g_k^{-\a}\ra$ 
% is a non-zero multiple of 
% $$-\frac{\l_{i_k}^{-1}-1}{\l_{i_{m+1}}^{-1}-1}\quad 
% (i_k\ne p,q,m+2).
% $$
% for $0\le k\le m+2$, $k\ne \imath(p),\imath(q),\imath(m+2)$.
% In each case, the intersection matrix is non-degenerate, 
% and $m$ elements $\breve{\ell}_{\infty,j}^\a$ are linearly independent. 
% We investigate the relations of these eigenspaces in the following five cases.

\medskip\noindent
Case 1: $\a_p,\a_q, \a_p+\a_q\notin \Z$.\\
We may assume that  $\imath(p)\ne m+1$ since   
either $\imath(p)\ne m+1$ or $\imath(q)\ne m+1$ holds. 
We extend the intersection matrix $H'_{m-1}$ to $H'_{m}$ by adding the 
$\breve{\ell}_{\infty,m+1}^\a$ and $\g_{\imath(p)}^{-\a}$. 
%to the last row and columns. 
We can choose a branch of $u_x$ on $\breve{\ell}_{\infty,m+1}$ so that 
$$
\la \breve{\ell}_{\infty,m+1}^\a, \g_{\imath(p)}^{-\a}\ra
=-\frac{\l_{p}^{-1}-1}{\l_{i_{m+1}}^{-1}-1}\ne 0.
$$
Since 
$$
\la \breve{\ell}_{\infty,j}^\a, 
\g_{\imath(p)}^{-\a}\ra=0 
$$
for $j\in \{0,1,\dots,m,m+2\} -\{\imath(p),\imath(q),\imath(m+2)\}$,  
$H'_m$ is lower triangle and $\det(H_m)\ne 0$. Thus we have 
$\dim E_{pq}^\a(1)=m$.
% On the other hand, by (\ref{eq:intno-DV}) and $\l_p\l_q\ne 1$, we have
% $\la \ell_{\imath(p)\imath(q)}^\a ,\g_{\imath(p)\imath(q)}^{-\a}\ra\ne 0$,  
% which implies that $\ell_{\imath(p)\imath(q)}^\a$ is different from 
% the zero element. 
% Thus the space $\lftH{\a}$ is spanned by $E_{pq}^\a(1)$ and 
% the eigenvector ${\ell}^\a_{\imath(p)\imath(q)}$ of $\cM^\a_{pq}$ of 
% eigenvalue $\l_p\l_q\ne 1$.
%Since the assumption of Remark \ref{rem:cancel} is satisfied in this case,
%Hence the eigenspaces of $\cM^\a_{pq}$ coincide with those with 
%the expression (\ref{eq:cM-exp}) by Remark \ref{rem:cancel}, and 
By Remark \ref{rem:cancel}, 
$\cM^\a_{pq}$ admits the expression (\ref{eq:cM-exp}).

% We can similarly show that $\ell_{\imath(p)\imath(q)}^{\a}$ is 
% a $\l_p\l_q$-eigenvector of $\cM^\a_{pq}$, and that the $1$-eigenspace 
% $\cM^\a_{pq}$ is $m$-dimensional. Thus $\cM^\a_{pq}$ admits 
% the expression  (\ref{eq:cM-exp}). 
% To obtain the expression  (\ref{eq:cN-exp}), we have only to note  
% that $\g_{\imath(p)\imath(q)}^{-\a}$ 
% is a $\l_p^{-1}\l_q^{-1}$-eigenvector of $\cM^\a_{pq}$. 
% We have the expressions (\ref{eq:cM-exp}) and (\ref{eq:cN-exp}) 
% by the same way as in Case 2. 

\medskip\noindent
Case 2: $\a_p\notin \Z$, $\a_q\in \Z$ or $\a_p\in \Z$, $\a_q\notin \Z$.\\
We show the former. 
Since $\a_p+\a_q\notin \Z$, ${\ell}_{\imath(p)\imath(q)}^\a$ is 
an eigenvector $\cM^\a_{pq}$ of eigenvalue $\l_p\l_q\ne1$. 
If $p\ne i_{m+1}$ then  we can show that 
the space $\lftH{\a}$ is spanned by $E_{pq}^\a(1)$ and 
the eigenvector ${\ell}^\a_{\imath(p)\imath(q)}$ by the same way as in Case 1. 
If $p=i_{m+1}$ then the intersection matrix 
$$\big(\la \breve{\ell}_{\infty,j}^\a, 
\g_k^{-\a}\ra\big)_{j,k}
\quad (j,k\in \{0,1,\dots,m,m+2\} -\{\imath(q),\imath(m+2)\})$$ 
becomes a diagonal matrix of size $m$ with non-zero diagonal entries, which 
implies that $\dim E_{pq}^\a(1)=m$.
Thus $\cM^\a_{pq}$ admits the expression (\ref{eq:cM-exp}).
% and (\ref{eq:cN-exp}). 

\medskip\noindent
Case 3: $\a_p,\a_q\notin \Z$, $\a_p+\a_q\in \Z$, $r<m$.\\
Since $r<m$, there exists $0\le k\le m$ such that $\a_{i_k}\not \in \Z.$ 
We can reconstruct bases of $\lftH{\a}$ and $\tH{-\a}$ so that  we 
choose the index $i_{m+1}$ satisfying $i_{m+1}\ne p,q$. 
We can show that $\dim E_{pq}^\a(1)=m$ by the same way as in Case 1.
In this case,  ${\ell}^\a_{\imath(p)\imath(q)}$ satisfies 
%is $1$-eigenvector of $\cM^\a_{pq}$ 
$$\cM^\a_{pq}({\ell}^\a_{\imath(p)\imath(q)})
={\ell}^\a_{\imath(p)\imath(q)},\quad 
\la \ell_{\imath(p)\imath(q)}^\a ,\g_{\imath(p)\imath(q)}^{-\a}\ra= 0$$
by $\l_p\l_q= 1$, (\ref{eq:intno-DV}) and (\ref{eq:non-1-EV}).
%Since $r<m$, there exists $0\le k\le m$ such that $i_k\ne p,q$ and 
%$\a_{i_k}\not \in \Z.$ 
Since $i_{m+1}\ne \imath(p),\imath(q)$, we have
$$
\la \ell_{\imath(p)\imath(q)}^\a,\g_{\imath(p)}^{-\a} \ra=-\frac{1}{\l_p}
\ne 0,
\quad 
\la \ell^{\a}_{\imath(p)},\g_{\imath(p)\imath(q)}^{-\a}\ra=
\frac{\l_{q}-1}{\l_{q}}
\ne 0.
$$
% if $\imath(q)=i_{m+1}$ then it does not hold.
Then neither $\ell_{\imath(p)\imath(q)}^{\a}$ nor 
$\g_{\imath(p)\imath(q)}^{-\a}$ is the zero element.
Thus $\ell_{\imath(p)\imath(q)}^{\a}$ is a $1$-eigenvector 
of $\cM^\a_{pq}$ and $(\g_{\imath(p)\imath(q)}^{-\a})^\perp$ is 
$m$-dimensional, and we have 
$$ 0\ne \ell_{\imath(p)\imath(q)}^{\a}\in E_{pq}^\a(1)=
(\g_{\imath(p)\imath(q)}^{-\a})^\perp\subsetneqq \lftH{\a}.$$
% In this case, $\ell_{\imath(p)\imath(q)}^{\a}$ is also 
% a $1$-eigenvector of $\cM^\a_{pq}$ satisfying  
% $$\la \ell_{\imath(p)\imath(q)}^{\a},\g_{\imath(p)\imath(q)}^{-\a}\ra 
% =0. 
% $$
Since the space spanned by eigenvectors of $\cM^\a_{pq}$ 
is $m$-dimensional,  $\cM^\a_{pq}$ is not diagonalizable. 
To characterize $\cM^\a_{pq}$, we have only to know the image 
of an element $\ell^\a$ satisfying 
$\la \ell^{\a},\g_{\imath(p)\imath(q)}^{-\a}\ra \ne 0$, 
since this $\ell^{\a}$ does not belong to 
$(\g_{\imath(p)\imath(q)}^{-\a})^\perp$ which coincides with the 
$1$-eigenspace of $\cM^\a_{pq}$. 
%by Remark \ref{rem:cancel}.
%since any element satisfying 
%$\la \ell^{\a},\g_{\imath(p)\imath(q)}^{-\a}\ra = 0$ 
%is invariant under the expression (\ref{eq:cM-exp}).
% The element $\ell^\a=\ell_{\imath(p)}^\a$ satisfies 
% $$\la \ell^{\a}_{\imath(p)},\g_{\imath(p)\imath(q)}^{-\a}\ra 
% =\frac{\l_{q}-1}{\l_{q}}
% \ne 0,$$
%and its image  under the expression (\ref{eq:cM-exp}) is 
By using the evaluated value of 
$\la \ell^{\a}_{\imath(p)},\g_{\imath(p)\imath(q)}^{-\a}\ra$, 
we obtain the image of $\ell^\a=\ell_{\imath(p)}^\a$ 
under the expression (\ref{eq:cM-exp}) as 
$$\ell^{\a}_{\imath(p)}-\l_{p}(\l_{q}-1)
\ell^\a_{\imath(p)\imath(q)}.$$
% to check by Maple, we express it by $\ell^\a_{\imath(p)}$ and $\ell^\a_{\imath(q)}$
% $$=(2-\l_{\imath(p)})\ell^\a_{\imath(p)}+(\l_{\imath(p)}-1)\ell^\a_{\imath(q)}.$$
On the other hand, it is easy to see that 
the continuation of $\ell^{\a}_{\imath(p)}$ along 
$\rho_{pq}$ is added 
$\l_{p}(1-\l_{q})\ell^\a_{\imath(p)\imath(q)}$ 
to it. Hence $\cM^\a_{pq}$ admits the expression (\ref{eq:cM-exp}).
%Similarly we have the expression (\ref{eq:cN-exp}) for $\cN_{pq}^\a$.

\medskip\noindent
Case 4: $\a_p,\a_q\notin \Z$, $\a_p+\a_q\in \Z$, $r=m$.\\
In this case, we have $\{p,q\}=\{i_{m+1},i_{m+2}\}$, $x_{i_0}=-\infty$, and 
$\g_k^{-\a}=(\circlearrowleft_{i_{k}},u^{-1}_x)$ $(k=0,1,\dots,m)$.
It is obvious that $\g_k^{-\a}$ are invariant under the deformation 
along $\rho_{pq}$. Thus not only $\cN^{-\a}_{pq}$ but also 
$\cM^\a_{pq}$ is the identity. 
On the other hand, 
$\ell_{\imath(p)\imath(q)}^{\a}$ is homologous to the zero element,  
since $\la \ell_{\imath(p)\imath(q)}^{\a},\g_k^{-\a}\ra=0$ for 
$k=0,1,\dots,m$. Hence each of the expressions (\ref{eq:cM-exp}) and 
(\ref{eq:cN-exp}) is the identity.
In this case, $E_{pq}^\a(1)$ is $m$-dimensional by the argument 
in Case 2 for $p=i_{m+1}$.
% , and a $1$-eigenvector of $\cM^\a_{pq}$ 
% not in $E_{pq}^\a(1)$ is given  by 
% $\breve{\ell}_{\infty,\imath(p)}^\a$, which is defined 
% by an arc $\breve{\ell}_{\infty,\imath(p)}$
% in the lower half space from $x_{m+2}=\infty$ to $x_{p}$ and 
% branches of $u_x$ on them.

\medskip\noindent
Case 5: $\a_p,\a_q\in \Z$.\\ 
Since (\ref{eq:non-1-EV}), $\l_p\l_q=1$ and 
$\la \ell_{\imath(p)\imath(q)}^{\a},\g_{\imath(q)}^{-\a}\ra\ne 0$, 
$\ell_{\imath(p)\imath(q)}^{\a}$ is a $1$-eigenvector of $\cM^\a_{pq}$. 
In this case, $\imath(p),\imath(q)<r<m+1$  and 
we may not extend the intersection matrix $H'_{m-1}$ to 
$H'_{m}$ by adding the $\breve{\ell}_{\infty,m+1}^\a$ and 
$\g_{\imath(p)}^{-\a}$ or $\g_{\imath(q)}^{-\a}$ as in Case 1,  
since 
$$
\la \breve{\ell}_{\infty,m+1}^\a, \g_{\imath(p)}^{-\a}\ra
=\la \breve{\ell}_{\infty,m+1}^\a, \g_{\imath(q)}^{-\a}\ra
=0.
$$
%It happens that $\dim E_{pq}^\a(1)=m-1$. 
However, we have another $1$-eigenvector
of $\cM^\a_{pq}$ not in $E_{pq}^\a(1)$.  
As seen in Case 3, the continuation of 
$\ell^{\a}_{\imath(p)}$ along $\rho_{pq}$ is 
$\ell^{\a}_{\imath(p)}+\l_{p}(1-\l_{q})\ell^\a_{\imath(p)\imath(q)}$. 
Since $\l_q=1$ and $\la \ell_{\imath(p)}^\a,\g_{\imath(p)}^{-\a}\ra\ne 0$,  
$\ell^{\a}_{\imath(p)}$ is a $1$-eigenvector of $\cM^\a_{pq}$.
Hence the $1$-eigenspace of $\cM^\a_{pq}$ is spanned by $E_{pq}^\a(1)$, 
$\ell^{\a}_{\imath(p)}$ and 
$\ell^{\a}_{\imath(q)}=\ell_{\imath(p)\imath(q)}^{\a}+\ell^{\a}_{\imath(p)}$.
Since it coincides with $\lftH{\a}$, $\cM^\a_{pq}$ is the identity. 
On the other hand, each of the expressions  (\ref{eq:cM-exp}) and 
(\ref{eq:cN-exp}) reduces to the identity, 
since $\g_{\imath(p)\imath(q)}^{-\a}$ 
degenerates to the zero element in this case by its definition.  
% Thus Hence in this case, $\cM^\a_{pq}$ and $\cN_{pq}^\a$ 
% coinside with the expressions  (\ref{eq:cM-exp}) and (\ref{eq:cN-exp}), 
% respectively.
\end{proof}

\section{Circuit Matrices}
\label{sec:c-Matrix}
Let $M_{pq}^\a$ and $N_{pq}^{-\a}$ be the representation matrix  
of $\cM^\a_{pq}$ with respect to the basis 
$\tr(\ell_0^\a,\ell_1^\a,\dots,\ell_m^\a)$ of $\lfSV{\a}$ 
and 
that of $\cN^{-\a}_{pq}$ with respect to $(\g_0^{-\a},\g_1^{-\a},\dots,
\g_m^{-\a})$ of $\SV{-\a}$. 
That is, the bases $\tr(\ell_0^\a,\ell_1^\a,\dots,\ell_m^\a)$ and 
$(\g_0^{-\a},\g_1^{-\a},\dots,\g_m^{-\a})$ are transformed into 
$$M_{pq}^\a \tr(\ell_0^\a,\ell_1^\a,\dots,\ell_m^\a), \quad 
(\g_0^{-\a},\g_1^{-\a},\dots,\g_m^{-\a})N_{pq}^{-\a}$$
by the continuation along $\rho_{pq}$.
We give their explicit forms. 

\begin{cor}
\label{cor:rep-mat}
We have 
\begin{align}
\label{eq:Mpq}
M_{pq}^\a&=I_{m+1}-\l_p\l_q  H \bw^{-\a}_{\imath(p)\imath(q)} \bv ^\a_{\imath(p)\imath(q)},\\
\label{eq:Npq}
N_{pq}^{-\a}&=I_{m+1}+ \bw^{-\a}_{\imath(p)\imath(q)} \bv ^\a_{\imath(p)\imath(q)}H,
\end{align}
where  $H$ is 
the intersection matrix given in (\ref{eq:int-mat}), 
the row vector $\bv ^\a_{\imath(p)\imath(q)}$ and the column vector 
$\bw^{-\a}_{\imath(p)\imath(q)}$ are expressed as
linear combinations 
$$\bv ^\a_{\imath(p)\imath(q)}=\ex_{\imath(q)}-\ex_{\imath(p)},\quad 
\quad \bw^{-\a}_{\imath(p)\imath(q)}
=(\l_p^{-1}-1)\ex^*_{\imath(q)}-(\l_q^{-1}-1)\ex^*_{\imath(p)}.
$$
Here $\ex_k$ $(k=0,1,\dots,m)$ are the unit row vectors 
$$\ex_0=(1,0,\dots,0),\ \ex_1=(0,1,0,\dots,0),\ \dots,\ 
\ex_m=(0,\dots,0,1)$$ 
of size $m+1$, $\ex_{m+1}=(0,0,\dots,0)$ and 
$\ex_{m+2}$ is given in (\ref{eq:e(m+2)}), 
$\ex_k^*=\tr\ex_k$ for $k=0,1,\dots,m,m+1$, and 
$\ex_{m+2}^*$ is given in (\ref{eq:e*(m+2)}).
They satisfy 
\begin{equation}
\label{eq:H-inv}
M_{pq}^\a\; H\; N_{pq}^{-\a}=H.
\end{equation}
\end{cor}

\begin{proof}
The spaces $\lfSV{\a}$ and $\SV{-\a}$ are identified with $\C^{m+1}$ by 
\begin{align*}
\lfSV{\a}\ni \ell^\a =
(v_0,v_1,\dots,v_m)\!\tr(\ell_0^\a,\ell_1^\a,\dots,\ell_m^\a)
% \begin{pmatrix}
% \ell_0^\a\\
% \ell_1^\a\\
% \vdots\\
% \ell_m^\a
% \end{pmatrix}
&\!\leftrightarrow\! \bv ^\a\!=\!(v_0,v_1,\dots,v_m)\in \C^{m+1},\\
\SV{-\a}\!\ni\! \g^{-\a}\! =\!
(\g_0^{-\a},\g_1^{-\a},\dots,\g_m^{-\a})\!\tr(w_0,w_1,\dots,w_m)
% \begin{pmatrix}
% w_0\\
% w_1\\
% \vdots\\
% w_m
% \end{pmatrix}
&\!\leftrightarrow \!
% \begin{pmatrix}
% w_0\\
% w_1\\
% \vdots\\
% w_m
% \end{pmatrix}
\bw^{-\a}\!=\!\tr(w_0,w_1,\dots,w_m) \in \C^{m+1}.
\end{align*}
We have 
$$
\begin{array}{cc}
\ell^\a_{\imath(p)\imath(q)}=\bv ^\a_{\imath(p)\imath(q)}
\tr(\ell_0^\a,\ell_1^\a,\dots,\ell_m^\a),
&
\g^{-\a}_{\imath(p)\imath(q)}=(\g_0^{-\a},\g_1^{-\a},\dots,\g_m^{-\a})
\; \bw^{-\a}_{\imath(p)\imath(q)}, 
\\[3mm]
\la \ell^\a,\g^{-\a}_{\imath(p)\imath(q)}\ra=\bv ^\a H \bw^{-\a}_{\imath(p)\imath(q)},
&
\la \ell_{\imath(p)\imath(q)}^\a,\g^{-\a}\ra=\bv ^\a_{\imath(p)\imath(q)}H\bw^{-\a},
\end{array}
$$
which imply that
\begin{align*}
\cM^\a_{pq}(\ell^\a)\!&=\!\ell^\a\!-\!\l_p\l_q
\la \ell^\a,\g^{-\a}_{\imath(p)\imath(q)}\ra\ell_{\imath(p)\imath(q)}^\a\!=\!
(\bv ^\a\!-\!\l_p\l_q \bv ^\a H\bw^{-\a}_{\imath(p)\imath(q)}
\bv ^\a_{\imath(p)\imath(q)})
\tr(\ell_0^\a,\ell_1^\a,\dots,\ell_m^\a),\\
\cN^{-\a}_{pq}(\g^{-\a})\!&=\!\g^{-\a}\!+\!
\la \ell_{\imath(p)\imath(q)}^\a,\g^{-\a}\ra\g^{-\a}_{\imath(p)\imath(q)}=
(\g_0^{-\a},\g_1^{-\a},\dots,\g_m^{-\a})
(\bw^{-\a}\!+\!\bw_{\imath(p)\imath(q)}^{-\a}
\bv ^\a_{\imath(p)\imath(q)} H \bw^{-\a}).
\end{align*}
Since 
$$\cM^\a_{pq}(\ell^\a)=\bv ^\a M^\a_{pq}
\tr(\ell_0^\a,\ell_1^\a,\dots,\ell_m^\a),\quad 
\cN^{-\a}_{pq}(\g^{-\a})=
(\g_0^{-\a},\g_1^{-\a},\dots,\g_m^{-\a})N^{-\a}_{pq}\bw^{-\a},
$$
% $v^\a \cdot M^\a_{pq}$ and $N^{-\a}_{pq}\cdot \bw^{-\a}$ are expressed as 
% $$v^\a(I_{m+1}-\l_p\l_qHH\bw^{-\a}_{pq} v^\a_{\imath(p)\imath(q)}),\quad 
% (I_{m+1}-\bw_{\imath(p)\imath(q)}^{-\a}
% \bv ^\a_{\imath(p)\imath(q)} H) \bw^{-\a},$$
% respectively.
the expressions (\ref{eq:Mpq}) and (\ref{eq:Npq}) are obtained. 
The equality (\ref{eq:H-inv}) is a consequence from  
Proposition \ref{prop:keep-int} (\ref{item:invariant}). We can also show it 
by a direct computation using 
$\la \ell_{\imath(p)\imath(q)}^\a ,\g_{\imath(p)\imath(q)}^{-\a}\ra=
\bv^\a_{\imath(p)\imath(q)} H \bw^{-\a}_{\imath(p)\imath(q)}$
and (\ref{eq:intno-DV}).
\end{proof}

\section{Examples}
\label{sec:example}
We give examples of circuit matrices. We set 
$m=3$, $r=2$, $\a_0,\a_1\in \Z$, $\a_2\dots,\a_5\notin \Z$, $\a_2+\a_3\in \Z$, 
%$x_{5}=\infty$, 
and $\imath$ is the identical permutation. The circuit matrices $M_{pq}^\a$ 
and $N_{pq}^{-\a}$ are given as follows:
\begin{alignat*}{3}
&M_{01}^\a=\left[ \begin {array}{cccc} 1&0&0&0\\ \noalign{\medskip}0&1&0&0
\\ \noalign{\medskip}0&0&1&0\\ \noalign{\medskip}0&0&0&1\end {array}
 \right], \quad &&
N_{01}^{-\a}= \left[ \begin {array}{cccc} 1&0&0&0\\ \noalign{\medskip}0&1&0&0
\\ \noalign{\medskip}0&0&1&0\\ \noalign{\medskip}0&0&0&1\end {array}
 \right], 
\\ 
%\noalign{\bigskip}
%
&M_{02}^\a= \left[ \begin {array}{cccc} \l_2&0&1-\l_2&0
\\ \noalign{\medskip}0&1&0&0\\ \noalign{\medskip}0&0&1&0
\\ \noalign{\medskip}0&0&0&1\end {array} \right], 
\quad&& N_{02}^{-\a}=
 \left[ \begin {array}{cccc} \l_2^{-1}&0&{\frac {\l_2-1}{
\l_2}}&0\\ \noalign{\medskip}0&1&0&0\\ \noalign{\medskip}0&0&1&0
\\ \noalign{\medskip}0&0&0&1\end {array} \right], 
\\ 
&M_{03}^\a= \left[ \begin {array}{cccc} \l_3&0&0&1-\l_3
\\ \noalign{\medskip}0&1&0&0\\ \noalign{\medskip}0&0&1&0
\\ \noalign{\medskip}0&0&0&1\end {array} \right], 
\quad &&N_{03}^{-\a}= \left[ \begin {array}{cccc} \l_3^{-1}&0&{
\frac {\left(\l_2-1\right)  \left(\l_3 -1 \right) }{\l_2\l_3}}&{
\frac {\l_3-1}{\l_3}}\\ \noalign{\medskip}0&1&0&0
\\ \noalign{\medskip}0&0&1&0\\ \noalign{\medskip}0&0&0&1\end {array}
 \right], 
\\ 
&M_{12}^\a= \left[ \begin {array}{cccc} 1&0&0&0\\ \noalign{\medskip}0&\l_2&1-
\l_2&0\\ \noalign{\medskip}0&0&1&0\\ \noalign{\medskip}0&0&0&1
\end {array} \right], 
\quad &&
N_{12}^{-\a}=
 \left[ \begin {array}{cccc} 1&0&0&0\\ \noalign{\medskip}0&\l_2^{-1}
&{\frac {\l_2-1}{\l_2}}&0\\ \noalign{\medskip}0&0&1&0
\\ \noalign{\medskip}0&0&0&1\end {array} \right], 
\\ 
&M_{13}^\a=  \left[ \begin {array}{cccc} 1&0&0&0\\ \noalign{\medskip}0&\l_3&0&
1-\l_3\\ \noalign{\medskip}0&0&1&0\\ \noalign{\medskip}0&0&0&1
\end {array} \right], 
\quad &&
N_{13}^{-\a}=
 \left[ \begin {array}{cccc} 1&0&0&0\\ \noalign{\medskip}0&\l_3^{-1}
&{\frac { \left( \l_2-1 \right)  \left( \l_3-1 \right) 
}{\l_2\l_3}}&{\frac {\l_3-1}{\l_3}}
\\ \noalign{\medskip}0&0&1&0\\ \noalign{\medskip}0&0&0&1\end {array}
 \right],  
\\ 
&M_{14}^\a=  \left[ \begin {array}{cccc} 1&0&0&0\\ \noalign{\medskip}0&\l_4&0&0
\\ \noalign{\medskip}0&0&1&0\\ \noalign{\medskip}0&0&0&1\end {array}
 \right], 
\quad &&
N_{14}^{-\a}=
 \left[ \begin {array}{cccc} 1&0&0&0\\ \noalign{\medskip}0&\l_4^{-1}
&{\frac {1-\l_2}{\l_2\l_4}}&{\frac {1-\l_3}{
\l_3\l_4}}\\ \noalign{\medskip}0&0&1&0\\ \noalign{\medskip}0
&0&0&1\end {array} \right], 
\\ 
&M_{23}^\a=  \left[ \begin {array}{cccc} 1&0&0&0\\ \noalign{\medskip}0&1&0&0
\\ \noalign{\medskip}0&0&2-\l_2&\l_2-1\\ \noalign{\medskip}0&0
&1-\l_2&\l_2\end {array} \right], 
\quad &&
N_{23}^{-\a}=
 \left[ \begin {array}{cccc} 1&0&0&0\\ \noalign{\medskip}0&1&0&0
\\ \noalign{\medskip}0&0&{\frac {2\l_2-1}{\l_2}}&1-\l_2
\\ \noalign{\medskip}0&0&{\frac {\l_2-1}{\l_2^{2}}}&\l_2^{-1}\end {array} \right], 
\\ 
&M_{24}^\a=  \left[ \begin {array}{cccc} 1&0&\l_2-1&0\\ \noalign{\medskip}0&1
&\l_2-1&0\\ \noalign{\medskip}0&0&\l_2\l_4&0
\\ \noalign{\medskip}0&0&\l_4 \left( \l_2-1 \right) &1
\end {array} \right], 
\quad &&
N_{24}^{-\a}=
 \left[ \begin {array}{cccc} 1&0&0&0\\ \noalign{\medskip}0&1&0&0
\\ \noalign{\medskip}0&0&{\frac {1}{\l_2\l_4}}&{\frac {1-
\l_3}{\l_3\l_4}}\\ \noalign{\medskip}0&0&0&1\end {array}
 \right], 
\\ 
&M_{34}^\a=\left[ \begin {array}{cccc} 1&0&0&\l_3-1\\ \noalign{\medskip}0&1
&0&\l_3-1\\ \noalign{\medskip}0&0&1&\l_3-1
\\ \noalign{\medskip}0&0&0&\l_3\l_4\end {array} \right], 
\quad &&
N_{34}^{-\a}=
 \left[ \begin {array}{cccc} 1&0&0&0\\ \noalign{\medskip}0&1&0&0
\\ \noalign{\medskip}0&0&1&0\\ \noalign{\medskip}0&0&{\frac {1-
\l_2}{\l_2}}&{\frac {1}{\l_3\l_4}}\end {array}
 \right].
\end{alignat*}

\section{Identification of $\lfSV{\a}$ and $\Sol_x(a,b,c)$}
\label{sec:idetify}
We define a holomorphic $1$-from $\w$ on $T_x$ by 
$$\w=d\log u_x(t)=\sum_{i=0}^{m+1} \frac{\a_i}{t-x_i}dt.$$
A twisted cohomology group and 
that with compact support are defined by 
\begin{equation}
\label{eq:twisted-cohomology}
\begin{array}{rcl}
\tC{\a}\!&\!=\!&\! \ker(\na_{\w}:\cE^1(T_x)\to \cE^2(T_x))/\na_{\w}
(\cE^0(T_x)),\\[2mm]
\ctC{\a}\!&\!=\!&\! \ker(\na_{\w}:\cE_c^1(T_x)\to \cE_c^2(T_x))/\na_{\w}
(\cE_c^0(T_x)),
\end{array}
\end{equation}
respectively, where $\cE^k(T_x)$ and $\cE_c^k(T_x)$ are  
the vector space of smooth $k$-forms on $T_x$ and that 
with compact support, and $\na_\w$ is a twisted exterior derivative 
$d+\w\wedge$. 
%These spaces are $(m+1)$-dimensional and 
There is the  natural map $\jmath^\a_c:\ctC{\a} \to \tC{\a}$ by 
the inclusion $\cE_c^k(T_x)\hookrightarrow \cE^k(T_x)$. 
We can regard $\ctC{\a}$ and $\tC{\a}$ as the dual spaces of 
$\lftH{\a}$ and $\tH{\a}$  via the pairings 
$$
\la \ell^{\a},\psi_c\ra=\sum_{\mu} \int_{\sigma_\mu} u(t,x)\psi_c, 
 \quad 
\la \g^{\a},\psi\ra=\sum_{\nu} \int_{\tau_\nu}
u(t,x)\psi,
$$
where $\ell^{\a}=\sum_{\mu}(\sigma_\mu,u_x)\in\lftH{\a}$, $\psi_c\in \ctC{\a}$ 
and $\g^{\a}=\sum_{\nu} (\tau_\nu,u_x)\in \tH{\a}$, $\psi\in \tC{\a}$. 
Period matrices  $\varPi_c^{lf}(\a,x)$ and $\varPi(\a,x)$ are defined by 
the pairings of bases of $\lftH{\a}$ and $\ctC{\a}$, 
and those of $\tH{\a}$ and $\tC{\a}$, respectively. 
Under the condition (\ref{eq:old-condition}), we can construct 
$\varPi_c^{lf}(\a,x)$ and $\varPi(\a,x)$ so that 
each column vector of them is a fundamental system of solutions 
to $\cF_D(a',b',c')$ for some $(a',b',c')$, 
of which difference from $(a,b,c)$ is an integral vector.
On the other hand, under the condition (\ref{eq:non-int}), 
it happens that $\varPi_c^{lf}(\a,x)$ and $\varPi(\a,x)$ do not include a 
fundamental system of solutions to $\cF_D(a,b,c)$. 
For $\varPi^{lf}_c(\a,x)$, there is a case when 
the form $dt/(t-1)$ in (\ref{eq:int-rep}) does not belong to the image of 
the natural map $\jmath^\a_c$. 
For $\varPi(\a,x)$, if $\a_k\in\Z$ and 
$u_x(t)dt/(t-1)$ is single valued holomorphic around $t=x_i$, 
then 
$$\la \g^\a_{\imath(k)},dt/(t-1) \ra= 
\int_{\circlearrowleft_k}u_x(t)\frac{dt}{t-1}=0.$$
In spite of this situation, 
$\cM^\a$ and $\cN^\a$ can be regarded as the monodromy representations of 
$\varPi_c^{lf}(\a,x)$ and $\varPi(\a,x)$, respectively, since the 
pairing between $\lftH{\a}$ and $\ctC{\a}$ and that between  
$\tH{\a}$ and $\tC{\a}$ are perfect, and the bases of 
$\ctC{\a}$ and $\tC{\a}$ can be extended to global frames of 
vector bundles over $X$ with fibers $\ctC{\a}$ and $\tC{\a}$.

In general,  the stalk of $\lfLS{\a}$ at $x$ cannot be regarded as 
the local solution space to $\cF_D(a,b,c)$ around $x$ 
under only the condition (\ref{eq:non-int}). 
Hence we need additional conditions to regard $\cM^\a$ as 
the monodromy representation of $\cF_D(a,b,c)$. 
Hereafter, we assume that there exist at least three 
non-integral parameters $\a_{i_m}$, $\a_{i_{m+1}}$ and $\a_{i_{m+2}}$ 
in $\a$. 
If $\a_{i_m}+\a_{i_{m+1}}=\a_{i_m}+\a_{i_{m+2}}= 
\a_{i_{m+1}}+\a_{i_{m+2}}=0$ then 
$\a_{i_m}=\a_{i_{m+1}}=\a_{i_{m+2}}=0$. Thus   
we have the condition (\ref{eq:our-condition}) in this case. 
%\begin{equation}
%\label{eq:add-condition}
%$$\a_{i_m}, \a_{i_{m+1}},\a_{i_{m+2}}\notin \Z,\quad 
%\a_{i_{m+1}}+\a_{i_{m+2}}\ne0. 
%$$
%\end{equation}
We shift $\a$ to 
\begin{equation}
\label{eq:new-parametar}
\Hat\a=\a +\cex_{i_{m+1}}+\cex_{i_{m+2}}-\cex_{m+1}-\cex_{m+2},
\end{equation}
where $\cex_0=(1,0,\dots,0)$, $\cex_1=(0,1,0,\dots,0)$, \dots, 
%$\cex_{m}=(0,\dots,1,0,0)$, $\ex_{m+1}=(0,\dots,0,1,0)$ and 
$\cex_{m+2}=(0,\dots,0,1)$ 
are the unit row vectors of size $m+3$.
Note that the condition (\ref{eq:cond-wa}) 
%$\sum_{i=0}^{m+2} \Hat\a_i=0.$
is also satisfied by  $\Hat\a$.
We have $\Hat u(t,x)$, $\Hat\w$, $\lftH{\pm\Hat\a}$, $\tH{\pm\Hat\a}$, 
$\ctC{\pm\Hat\a}$ and  $\tC{\pm\Hat\a}$ for the shifted $\Hat \a$.

% We set 
% %a meromorphic $1$-from $\Hat\w$ and a covariant derivative $\na_{\Hat\w}$ by 
% %a multi-valued function $\Hat u(t,x)$ on $\fX$ by 
% $$\Hat\w=\sum_{i=0}^{m+1} \frac{\Hat\a_i}{t-x_i}dt,\quad 
% \na_{\Hat\w}=d+\Hat\w\wedge
% %\Hat u(t,x)=\prod_{i=0}^{m+1} (t-x_i)^{\Hat\a_i},
% $$

% A twisted cohomology group and 
% that with compact support are defined by 
% \begin{align*}
% \tC{\Hat\a}&=\ker(\na_{\Hat\w}:\cE^1(T_x)\to \cE^2(T_x))/\na_{\Hat\w}
% (\cE^0(T_x)),\\
% \ctC{\Hat\a}&=\ker(\na_{\Hat\w}:\cE_c^1(T_x)\to \cE_c^2(T_x))/\na_{\Hat\w}
% (\cE_c^0(T_x)),
% \end{align*}
% respectively, where $\cE^k(T_x)$ and $\cE_c^k(T_x)$ are  
% the vector space of smooth $k$-forms on $T_x$ and that 
% with compact support. 
% where $d_t$ is the exterior derivative with respect to $t$ and  
% $x_0$ and $x_{m+1}$ are regarded as $0$ and $1$, respectively.
\begin{proposition}
\label{prop:compact-supp}
Suppose that 
%(\ref{eq:non-int}),  
(\ref{eq:our-condition}) and 
none of $\Hat\a_0,\dots,\Hat\a_{m+2}$ is a negative integer:
\begin{equation}
\label{eq:non-negative}
\Hat\a_i\notin \{-1,-2,-3,\dots\}\quad (0\le i\le m+2).
\end{equation}
The form 
$$\f=\frac{dt}{t-x_{i_{m+1}}}-\frac{dt}{t-x_{i_{m+2}}}$$
represents a non-zero element in both spaces 
$\ctC{\Hat\a}$ and $\tC{-\Hat\a}$.
%where these cohomology groups are defined by the shift $\a\mapsto \hat \a$ for $\ctC{\a}$ and $\tC{-\a}$.

\end{proposition}

\begin{proof} 
Under the assumption of this proposition, 
we can transform $\f$ into a cohomologous element $\f_c$ in $\cE_c^1(T_x)$ 
by following \cite[\S4]{M1}. The intersection form $\la \;, \ra$ 
between $\ctC{\Hat\a}$ and $\tC{-\Hat\a}$ is defined as 
$$\la \xi_c,\eta\ra=\int_{T_x} \xi_c\wedge \eta, \quad 
\xi_c\in \ctC{\Hat\a},\ \eta\in \tC{-\Hat\a}.$$
By using \cite[Theorem 2.1]{M1}, we have the intersection number 
$$\la \f_c,\f\ra=
2\pi\sqrt{-1} 
\frac{\Hat\a_{i_{m+1}}+\Hat\a_{i_{m+2}}}{\Hat\a_{i_{m+1}}\Hat\a_{i_{m+2}}},
$$
which does not vanish. 
\end{proof}

\begin{remark}
The condition 
%$(\ref{eq:our-condition})$ 
$(\ref{eq:our-condition})$ is essential. 
When $\a_0=\dots=\a_m=0$ and $\a_{m+2}=-\a_{m+1}\in \C-\Z$, 
we have $\Hat\a=\a$ and $\w=\a_{m+1}\dfrac{dt}{t-1}$.
Since $\na_{-\w}( 1)=-\a_{m+1}\f$, and $\f$ is the zero element of $\tC{-\a}$.
\end{remark}

\begin{theorem}
\label{th:idetify}
Under the conditions 
%(\ref{eq:non-int}), 
(\ref{eq:our-condition}) and (\ref{eq:non-negative}), 
the local solution space $\Sol_{x}(a,b,c)$ to 
$\cF_D(a,b,c)$ around $x$ is identified with 
the stalk of $\lfLS{\a}$ over $x$.  
% The monodromy representation of $\Sol_{\dot x}(a,b,c)$ around $\dot x$ 
% coincides with $\cM^\a$ studied in \S\ref{sec:monod-rep}.
\end{theorem}

\begin{proof}
Note that 
$$u(t,x)\frac{dt}{t-x_{m+1}}=\Hat u(t,x)\frac{\f}{x_{i_{m+1}}-x_{i_{m+2}}}
%,\quad \Hat u(t,x)=\prod_{i=0}^{m+1} (t-x_i)^{\Hat\a_i}
.$$ 
% The pairing between $\ctC{\Hat\a}$ and $\lftH{\Hat\a}$ is defined by 
% the integral 
% $$
% \la \psi, \ell^{\Hat\a}\ra=\sum_{\mu} \int_{\sigma_\mu}\Hat u(t,x)\psi, 
% $$
% where $\psi\in \ctC{\Hat\a}$ and 
% $\ell^{\Hat\a}=\sum_{\mu} (\sigma_\mu,\Hat u_x(t))\in \lftH{\Hat\a}$. 
Since $\f$ is cohomologous to $\f_c\in \ctC{\Hat\a}$ by Proposition 
\ref{prop:compact-supp}, 
if the improper integral
$$\int_{x_{i_{m+1}}}^{x_{i_{k}}} \Hat u(t,x)\f$$
converges, then it coincides with the pairing $\la \ell_k^{\Hat\a},\f_c\ra$, 
which belongs to $\Sol_x(a,b,c)$. 
We claim that $\la \ell_k^{\Hat\a},\f_c\ra$ ($k=0,1,\dots,m$)  
span $\Sol_x(a,b,c)$. 
Since $\ell_k^{\Hat\a}$'s are linearly independent, we have to show the 
kernel of the map 
$$
\dfc:\lfSV{\Hat\a}\ni \ell^{\Hat\a} \mapsto 
\la \ell^{\Hat\a},\f_c\ra\in \Sol_x(a,b,c)
$$
is zero.
Since $\la \ell_k^{\Hat\a},\f_c\ra$ is a non-zero solution to 
$\cF_D(a,b,c)$ around $\dot x$, 
$\la \ell_k^{\Hat\a},\f_c\ra $ and  
$$\pa_i\la \ell_k^{\Hat\a},\f_c\ra=\la 
\ell_k^{\Hat\a},(\pa_i - \frac{\a_i}{t-x_i})\cdot \f_c \ra\quad (1\le i\le m)
$$ 
are linearly independent.
Hence $\ctC{\Hat\a}$ is spanned by 
$$\f_c,\quad  \psi_i=(\pa_i - \frac{\a_i}{t-x_i})\cdot \f_c\quad \quad (1\le i\le m).
$$
If $\ker\dfc\ne0$, then there is a non trivial relations among 
$\la \ell_k^{\Hat\a},\f_c\ra$'s. 
Then the period matrix 
$\la \ell_k^{\Hat\a},\psi_j\ra$ for bases 
$\ell_0^{\Hat\a},\ell_1^{\Hat\a},\dots,\ell_m^{\Hat\a}$ of $\lftH{\Hat\a}$
and $\psi_0=\f_c,\psi_1,\dots,\psi_m$ of $\ctC{\Hat\a}$ 
degenerates, since the linear relation is preserved under the 
action of the Weyl algebra. 
This contradicts to the perfectness of the 
pairing between $\lftH{\Hat\a}$ and $\ctC{\Hat\a}$. 
\end{proof}

% The monodromy representation of $\Sol_{\dot x}(a,b,c)$ is 
% the homomorphism from $\pi_1(X,\dot x)$ to $GL(\Sol_{\dot x}(a,b,c))$ 
% defined by the analytic continuation of every element in $\Sol_{\dot x}(a,b,c)$
% along loops representing elements of $\pi_1(X,\dot x)$.
% The circuit transformation is defined by its image of a loop $\rho$. 
% By taking a basis of $\Sol_{\dot x}(a,b,c)$, 
% we represent the circuit transformation by the circuit matirx. 

% We define 
% the monodromy representation of $\Sol_{\dot x}(a,b,c)$, 
% the circuit transformation of $\Sol_{\dot x}(a,b,c)$ and 
% its circuit matrix  with respect to a basis of $\Sol_{\dot x}(a,b,c)$
% as similarly to those of $\lfSV{\a}$ given in \S \ref{sec:monod-rep}.

\begin{cor}
Under the conditions 
%(\ref{eq:non-int}), 
(\ref{eq:our-condition}) and (\ref{eq:non-negative}), 
the circuit transformation of $\Sol_{\dot x}(a,b,c)$ 
along the loop $\rho_{pq}$ coincides with $\cM^\a_{pq}$ in 
(\ref{eq:cM-exp}). 
Its circuit matrix  with respect to the basis 
$$\tr\big(\la \ell_0^{\Hat\a},\f_c\ra,\la \ell_1^{\Hat\a},\f_c\ra,\dots,
\la \ell_m^{\Hat\a},\f_c \ra\big)$$ 
coincides with $M^\a_{pq}$ in (\ref{eq:Mpq}).
\end{cor}

\begin{proof}
%To obtain the monodromy representation of $\Sol_{\dot x}(a,b,c)$, 
Since $\l_i=\exp(2\pi\sqrt{-1}\a_i)=\exp(2\pi\sqrt{-1}\Hat\a_i)$, 
there is a natural isomorphism between 
$\lftH{\a}$ and $\lftH{\Hat\a}$ given by the parameter shift $\a\to\Hat\a$. 
Theorem \ref{th:idetify} yields this corollary. 
\end{proof}

We shift $\a$ to 
\begin{equation}
\label{eq:negative-parametar}
\Check\a=\a -\cex_{i_{m+1}}-\cex_{i_{m+2}}+\cex_{m+1}+\cex_{m+2}.
\end{equation}
We have $\Check u(t,x)$, $\Check\w$, $\lftH{\pm\Check\a}$, $\tH{\pm\Check\a}$, 
$\ctC{\pm\Check\a}$ and  $\tC{\pm\Check\a}$ for the shifted $\Check \a$.

\begin{cor}
Under the conditions 
%(\ref{eq:non-int}), 
(\ref{eq:our-condition}) and 
\begin{equation}
\label{eq:non-negative-0}
\Check\a_i\notin \{0,-1,-2,-3,\dots\}\quad (0\le i\le m+2),
\end{equation}
the circuit transformation of $\Sol_{\dot x}(-a,-b,-c)$ 
along the loop $\rho_{pq}$ coincides with $\cN^{-\a}_{pq}$ in 
(\ref{eq:cN-exp}). 
Its circuit matrix  with respect to the basis 
$$\tr\big(\la \g_0^{-\Check\a},\f\ra,\la \g_1^{-\Check\a},\f\ra,\dots,
\la \g_m^{-\Check\a},\f\ra\big)$$ 
coincides with $\tr N^{-\a}_{pq}$ in (\ref{eq:Npq}).
\end{cor}

\begin{proof}
Since the condition (\ref{eq:non-negative}) is satisfied 
under (\ref{eq:non-negative-0}), $\f$ represents a non-zero element  
%in both spaces $\lftH{\Check\a}$ and 
of $\tH{-\Check\a}$.
% The pairing between 
% $\tC{-\Check\a}$ and $\tH{-\Check\a}$ is defined by 
% the integral 
% $$
% \la \psi', \g^{-\Check\a}\ra=\sum_{\nu} \int_{\tau_\nu}\frac{\psi'}
% {\Check u(t,x)}, \quad \Check u(t,x)=\prod_{i=0}^{m+1} (t-x_i)^{\Check \a},
% $$
% where $\psi'\in \tC{-\Check\a}$ and 
% $\g^{-\Check\a}=\sum_{\nu} (\tau_\mu,1/\Check u_x(t))\in \tH{-\Check\a}$. 
Since 
$$
%\prod_{i=0}^{m+1}(t-x_i)^{-\a_i}
\frac{1}{u(t,x)}\cdot 
\frac{dt}{t-1}
=
%\prod_{i=0}^{m+1}(t-x_i)^{-\Check\a_i}
\frac{1}{\Check u(t,x)}\cdot
\frac{\f}{x_{i_{m+1}}-x_{i_{m+2}}},\quad 
%\f
% =\frac{1}{x_{i_{m+1}}-x_{i_{m+2}}}\left(\frac{dt}{t-x_{i_{m+1}}}
% -\frac{dt}{t-x_{i_{m+2}}}\right)
%=\frac{dt}{(t-x_{i_{m+1}})(t-x_{i_{m+2}})}
$$
$\la  \g^{-\Check\a}_k,\f\ra$ $(k=0,1,\dots,m)$ belong to 
$\Sol_{\dot x}(-a,-b,-c)$ and they do not vanish identically under 
(\ref{eq:non-negative-0}).
By using the same argument as in Proof of Theorem \ref{th:idetify}, 
we can show that they span $\Sol_{\dot x}(-a,-b,-c)$.
By the identification of $\SV{-\Check\a}$ and $\Sol_{\dot x}(-a,-b,-c)$ 
together with a natural isomorphism between $\tH{-\a}$ and $\tH{-\Check\a}$, 
we have this corollary.
\end{proof}

\end{document}